\documentclass[11pt,a4paper]{article}
\usepackage{authblk}
\usepackage[french,english]{babel}
\usepackage[utf8]{inputenc}
\usepackage[T1]{fontenc}
\usepackage{amsmath}
\usepackage{amsfonts}
\usepackage{amssymb}
\usepackage{hyperref}
\hypersetup{
    colorlinks=true,
    urlcolor=blue,
    linkcolor=blue,
    citecolor=blue,
    breaklinks=true
}
\usepackage{url}
\usepackage{xcolor,color}
\colorlet{darkgreen}{green!50!black}
\usepackage{amsbsy}
\usepackage{amsmath}
\usepackage{dsfont}
\usepackage{epigraph}

\usepackage{geometry}
\usepackage{graphicx}

\usepackage[numbers]{natbib}

\usepackage{amsthm}
\usepackage{oubraces}

\usepackage{bm}

\newtheorem{theorem}{Theorem}

\newtheorem{lemma}[theorem]{Lemma}
\newtheorem{proposition}[theorem]{Proposition}
\newtheorem{corollary}[theorem]{Corollary}
\newtheorem{remark}[theorem]{Remark}

\author[$1$]{Sandro Franceschi}
\affil[$1$]{T\'el\'ecom SudParis, Institut Polytechnique de Paris}
\title{Martin boundary of a space-time Brownian motion with drift \\ killed at the boundary of a moving cone}
\date{}

\begin{document}

\maketitle
\thispagestyle{empty}

\abstract{
We study a space-time Brownian motion with drift $B(t)=(t_0+t,y_0+W(t)+\gamma t)$ killed at the moving boundary of the cone $\{(t,x): 0\leqslant x \leqslant t \}$. This article determines the parabolic Martin boundary and all harmonic functions associated with this process. To that end, the asymptotics of Green's functions are determined along all directions. We also find the exit probabilities at the edges, the probability of remaining in the cone forever and the laws of the exit point and exit time. From this, we derive an explicit formula for the transition kernel of the process.
These results arise from two different methods initially introduced to study random walks. An analytical approach, developed in the 1970s by Malyshev and based on the steepest descent method on a Riemann surface, is used to determine the asymptotics of the Green's functions. A recursive compensation approach, inspired by the method developed in the 1990s by Adan, Wessels and Zijm, is used to determine the harmonic functions.
}

\section{Introduction}

\paragraph{Main process and overview of the results}

Let $W(t)$ be a standard Brownian motion, $\gamma\in (0,1)$ a drift, $(t_0,y_0)$ a starting point such that $0<t_0<y_0$.
We define the space-time Brownian motion by
\begin{equation}
B(t):=(t_0+t,y_0+W(t)+\gamma t).
\label{eq:B}
\end{equation}
Let the cone
\begin{equation}
C:= \{ (t',y): 0<y<t' \}
\label{eq:C}
\end{equation}
which defines a two-sided moving boundary,
and $T$ the first exit time of the cone
\begin{equation}
T:=\inf \{t\geqslant 0 : B(t)\notin C \} .
\label{eq:T}
\end{equation}
We also define the exit times on each edge of the cone
\begin{equation}
T_1:=\inf \{ t\geqslant 0 : B(t)=(t_0+t,t_0+t) , \ t>0 \}
\quad\text{and}\quad
T_2:=\inf \{ t\geqslant 0 : B(t)=(t_0+t,0) , \ t>0 \} 
\label{eq:T1T2}
\end{equation}
and we have $T=T_1 \wedge T_2$.

The purpose of the present work is to determine the parabolic Martin boundary associated with the space-time Brownian motion killed at the boundary of the cone $C$ (Theorem~\ref{cor:martinboundary}). This result derives from the asymptotics of Green's functions (Theorem~\ref{prop:asympt}).
An explicit expression is given for all the associated harmonic functions (Theorem~\ref{prop:minimalharmonic}), that is the positive functions $u$ satisfying 
\begin{equation}
\begin{cases}
(\partial^2_y +\gamma \partial_y +\partial_t)u(t,y)=0
& \text{for all } (t,y)\in C,
\\
u(t,y)=0 & \text{for all } (t,y)\in \partial C.
\end{cases}
\label{eq:harmyt}
\end{equation}
We find the persistence probability $\mathbb{P}^\alpha(T=\infty)$ (Proposition~\ref{prop:conditionedescapeproba}) and the probabilities of exit on an edge $\mathbb{P}^\alpha(T_1<T_2)$ and $\mathbb{P}^\alpha(T_2<T_1)$ (Proposition~\ref{prop:L1explicit}) for the process conditioned to drift in the direction $\alpha$ via Doob's $h$-transform.
We also compute the laws of $T$, $T_1$, $T_2$ and of the law of the exit point of the process $B(T)$ when $T<\infty$ (Theorem~\ref{thm:f1f2}). Finally, this gives a new original approach to obtain the transition kernel of the process killed at time $T$ (Corollary~\ref{cor:transitionkernel}), which we note
\begin{equation}
p^{k,C}_{(t_0,y_0)}(t,y) \mathrm{d}y :=\mathbb{P}_{(t_0,y_0)}(B(t)=(t_0+t,\mathrm{d}y), T>t).
\label{eq:defpkC}
\end{equation}

\paragraph{Proof strategy and structure of the article}
Section~\ref{sec:analyticalprelim} sets out some fundamental analytical preliminaries for our study. For practical reasons we first transform via a simple linear application, the space-time Brownian motion in the cone $C$ into a degenerate Brownian motion in the quadrant $\mathbb{R}_+^2$.
Then, we find a kernel functional equation (Proposition~\ref{prop:eqfunc}) connecting the Laplace transform of the Green's functions of this new process and the Laplace transform of the exit densities on the axes.
Section~\ref{sec:martin} is devoted to the asymptotics of Green's functions.
In this section, the key element of the proofs is the application of the steepest descent method on a Riemann surface generated by the kernel. 
Section~\ref{sec:mart}
gives a probabilist interpretation of all the harmonic functions in the cone in terms of the persistence probabilities of some conditioned processes via Doob's $h$-transform. Then, the full Martin boundary is determined.
Section~\ref{sec:harmcompensation} 
is based on a recursive compensation approach used to determine explicit expressions for the harmonic functions. The minimal Martin boundary is shown to be homeomorphic to a portion of a certain parabola defined by the kernel. Persistence and exit probabilities are computed.
Finally, Section~\ref{sec:transitionkernel} is dedicated to the inversion of the Laplace transforms to obtain the law of the exit point and the transition kernel of the process.

\paragraph{Related literature}

The literature on Brownian motion and moving boundaries is very rich, from the 1960s to the present day. The study of the crossing probability of the boundary is at the heart of many of these articles. The boundary can be one-sided or two-sided, depending on the problems considered. We cannot claim to be exhaustive here and will simply cite a few articles related to our problem. In a seminal paper published in 1960 \cite{anderson_1960}, Anderson considers a Brownian motion between two non-parallel straight lines and determines the probability of crossing one of the two straight lines before the other for a finite or infinite time horizon.
In 1964 \cite{skorokhod_1964}, see \cite[page 306-315]{skorokhod_2004} for an English version, Skorokhod founds
the probability density of the escape location of a Lévy process that does not cross 
two parallel straight lines, which provides the law of the process killed at the boundary.
In 1967 \cite{shepp_1967} Sheep considered for the first time in the continuous setting a square-root boundary. This type of boundary will subsequently be the subject of numerous studies~\cite{kralchev_moving_2007}. In the 1970s until recently Novikov solved many interesting related issues \cite{novikov_1979,novikov_1999}. 

There are also links between moving boundary results and Brownian motion in Weil chambers and reflection groups. For example, Biane, Bougerol, and O’Connell \cite{biane_bougerol_oconnell_2005} studied the probability that a Brownian motion with drift stays forever in a Weyl chamber, also called the persistence probability. Recently, Defosseux \cite{defosseux_2016} studied space-time Brownian motion in an affine Weyl chamber. Some of the results of this last paper are found again in the present article using other methods.

Since the seminal work of Martin~\cite{martin1941}, many books \cite{Doob2001} and articles have studied Brownian motion and random walks in cones, the asymptotics of their Green's functions, their Martin boundary and the associated harmonic functions. In the discrete setting, examples include
the work of Malyshev, Kourkova~et~al. on the Martin boundary of random walks in a quarter plane
\cite{malyshev_asymptotic_1973,kurkova_martin_1998,
kurkova_malyshevs_2003,kourkova_random_2011} using a saddle point method on a Riemann surface generated by the kernel of a functional equation. This powerful technique used in the present paper has recently been developed with success in the continuous setting for reflected Brownian motion in wedges or half-planes \cite{franceschi_asymptotic_2016,franceschi_kourkova_petit_2023,
ernst_franceschi_asymptotic_2021}.

In the higher dimensional case, the famous article of Ney and Spitzer
\cite{ney_spitzer_66} computes the asymptotics of the Green's function of random walks with drift in $\mathbb{Z}^d$ and shows that the Martin boundary is homeomorphic to the unit sphere.
Many other interesting articles by Ignatiouk et al. deals with the Martin boundary of random walks in half-spaces and orthants \cite{ignatiouk_loree_2010,ignatiouk-robert2010,ignatiouk_robert2009}. Denisov and Wachtel study the tail asymptotics for the exit time of a multidimensional random walk in a cone~\cite{denisov_wachtel_2015}. Duraj et al. determine the asymptotic of the Green function for random walks without drift confined to multidimensional convex cones \cite{duraj_tarrago_raschel_wachtel}.
Garbit and Raschel study the survival probability of multidimensional random walks in pyramids \cite{garbit_raschel_pyramid_2023}.
We also mention the article of Hoang et al. on the construction of harmonic functions in wedges using Boundary value problems \cite{hoang_raschel_tarrago_constructing_22}.

In the continuous setting, several papers study Brownian motion in cones in higher dimensions and we mention a few of them here. De Blassie determines the distributions of first exit times of Brownian motion in cones and computes asymptotics \cite{deblassie_87}. Bañuelos and Smits study the asymptotic behavior of Brownian motion in cones and express their results with infinite series involving its transition density \cite{banuelos_smits_97}. Garbit and Raschel study the asymptotics of the tail distribution of the first exit time of Brownian motion with drift for a large class of cones
\cite{garbit_exit_brownian_2014}.

Finally, we mention the compensation approach developed by Adan, Wessels and Zijm \cite{adan_wessels_zijm_compensation_93} which is an inspiration for this paper. The recent article of Hoang et al. \cite{hoang_raschel_tarrago_harmonic_22} uses this method to compute discrete harmonic functions in the quadrant. The paper of Ichiba et al. \cite{FIKR_2023} also uses this approach to compute the invariant measure of a degenerate competing three-particle system.



\section{Analytical preliminaries}
\label{sec:analyticalprelim}

\paragraph{Killed degenerate Brownian motion 
in the quadrant}
A simple linear transformation maps the space-time Brownian motion $B(t)$ defined in~\eqref{eq:C} in the cone $C$ onto a degenerate Brownian motion with drift $Z(t)$ in the quadrant $\mathbb{R}_+^2$.
To that purpose, we define the linear transform $\ell$ by
$$
\ell(t',y):=(t'-y,y) 
\quad\text{and}\quad
\ell^{-1}(x,y):=(x+y,y).
$$
Thanks to $\ell$ we perform the change of variable
$$t'=t_0+t=x+y .$$
We denote $B(t)=(t_0+t,y_0+W(t)+\gamma t)=(t',Y(t))$ the coordinates of the space-time Brownian motion. We define the drifted degenerate Brownian motion by 
\begin{align*}
Z(t)=(X(t),Y(t))& :=\ell (B(t))
\\ & \ = (t_0-y_0 - W(t)+(1-\gamma )t , y_0+W(t)+\gamma t)
\\ & \ =z_0+\mathbf{W}(t)+\mu t 
\end{align*}
where $\mathbf{W}(t):=(-W(t),W(t))$ is a degenerate Brownian motion in $\mathbb{R}^2$ of covariance matrix 
$
\begin{pmatrix}
1 & -1
\\
-1 & 1
\end{pmatrix}
$
and we denote $\mu:=(1-\gamma,\gamma)\in\mathbb{R}_+^2$ the positive drift and $z_0=(x_0,y_0):=(t_0-y_0,y_0)\in\mathbb{R}_+^2$ the starting point.

We have $\ell(C)=(\mathbb{R}_+^*)^2$ where $C$ is defined in~\eqref{eq:C}. The stopping time $T$ defined in~\eqref{eq:T} is equal to
$$
T=\inf \{ t\geqslant 0 : Z(t) \notin (\mathbb{R}_+^*)^2 \}.
$$
It is the first exit time of the quadrant for the process $Z(t)$. The stopping times on each of the boundaries defined in~\eqref{eq:T1T2} are equal to
$$
T_1=\inf \{ t\geqslant 0 : Z(t) \in \{0\}\times \mathbb{R}_+ \}
\quad\text{and}\quad
T_2=\inf \{ t\geqslant 0 : Z(t) \in  \mathbb{R}_+\times \{0\} \}
$$
and we still have $
T=T_1\wedge T_2$.

\begin{figure}[hbtp]
\centering
\includegraphics[scale=0.7]{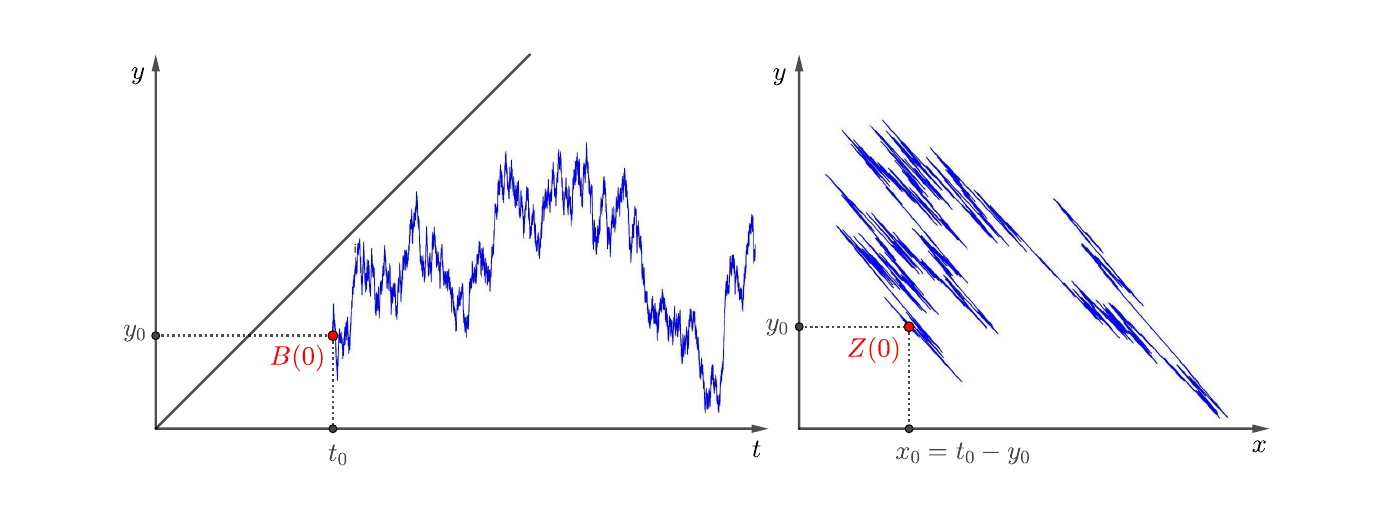}
\caption{Space time Brownian motion in $C$ and degenerate Brownian motion in $\mathbb{R}_+^2$}
\end{figure}

\paragraph{Green's functions and exit densities}

For a set $S\subset (\mathbb{R}_+^*)^2$ we define the Green's measure of the process killed at the boundary by
$$
G_{z_0}(S):=\int_0^\infty \mathbb{P}(Z(t)\in S ,t<T)\mathrm{d}t.
$$
Let us remark that for $S\subset (\mathbb{R}_+^*)^2$ we have $\mathbb{P}(Z(t)\in S ,t<T)=\mathbb{P}(Z(t\wedge T) \in S)$.
We assume that this measure has a density denoted
$$
g_{z_0}(z) \mathrm{d}z :=G_{z_0}(\mathrm{d}z).
$$
Let us recall that in~\eqref{eq:defpkC} we noted $p^{k,C}_{(t_0,y_0)}(t,y)$ the transition kernel of the space time Brownian motion killed in the cone $C$. It is important to remark that, up to the following change of variables, the Green's function of the process $Z$ killed at the boundary of the quadrant is equal to the transition kernel of $B$ the space time Brownian motion killed at the boundary of $C$ :
\begin{equation}
p^{k,C}_{(t_0,y_0)}(t,y)=g_{z_0}(z)
\quad
\text{ where } z_0=(t_0-y_0,y_0) \text{ and } z=(t_0+t-y,y).
\label{eq:green=densitytransition}
\end{equation}
One may consult the famous Doob's book  \cite[Chap IX \S 17]{Doob2001} which explains how Brownian Motion transition kernel may be seen as Green functions. 
The exit measures on the boundaries are defined for $S\subset \mathbb{R}_+^2$ by
\begin{align*}
A_1(S) &:=\mathbb{P}(Z(T_1)\in S , T_1<T_2)=\mathbb{E}(\mathds{1}_{Z(T_1)\in S}\mathds{1}_{T_1<T_2}),
\\
A_2(S) &:=\mathbb{P}(Z(T_2)\in S , T_2<T_1)
=\mathbb{E}(\mathds{1}_{Z(T_2)\in S}\mathds{1}_{T_2<T_1}),
\end{align*}
and the support of the measure $A_1$ (resp. $A_2$) lies on $\{0\}\times \mathbb{R}_+$ (resp. $\mathbb{R}_+\times \{0\}$).
We denote their densities $f_1$ and $f_2$ defined as follows
$$
f_1(y) \mathrm{d}y \delta_0(\mathrm{d}x) 
:=A_1(\mathrm{d}x,\mathrm{d}y),
\quad
f_2(x) \mathrm{d}x \delta_0(\mathrm{d}y) 
:=A_2(\mathrm{d}x,\mathrm{d}y),
$$
where $\delta_0$ is the Dirac measure in $0$.
One may notice that $(t_0+T)\mathds{1}_{T<\infty}=(X(T)+Y(T))\mathds{1}_{T<\infty}= Y(T_1) \mathds{1}_{T_1<T_2} + X(T_2) \mathds{1}_{T_2<T_1}$ and
we deduce that the cumulative distribution function of $T$ is equal to $F_T(t):=\mathbb{P}(T<t)=A_1(\{0\}\times [0,t_0+t])+A_2( [0,t_0+t]\times\{0\})$ and then $f_T$ the density of $T$ is equal to 
\begin{equation}
f_T(t)=f_1(t_0+t)+f_2(t_0+t).
\label{eq:fT}
\end{equation}
Note that to lighten the notation we have not noted the dependence at $z_0$ for $A_1$, $A_2$, $f_1$, $f_2$ and $f$.

\paragraph{Functional equation}
Let us define the Laplace transform of the Green's function for $(p,q)\in\mathbb{C}^2$ by
$$
L(p,q):=\iint_{\mathbb{R}_+^2} g_{z_0}(x,y)e^{px+qy} \mathrm{d}x\mathrm{d}y=\mathbb{E}\left(\int_0^\infty e^{(p,q)\cdot Z_t} \mathds{1}_{t<T} \mathrm{d}t \right) .
$$
Remembering that $Z=(X,Y)$ the Laplace transform of the exit densities are defined by
$$
L_1(q):=\int_0^\infty f_1(y) e^{q y}\mathrm{d}y=
\mathbb{E}\left(e^{q Y({T_1})} \mathds{1}_{T_1<T_2} \right) ,
\quad
L_2(p):=\int_0^\infty f_2(x) e^{p x}\mathrm{d}x
=\mathbb{E}\left(e^{p X({T_2})} \mathds{1}_{T_2<T_1} \right) .
$$
Once again, we omit to note the dependence at $z_0$ for $L$, $L_1$ and $L_2$. One may remark that $L(0,0)=\mathbb{E}[T]=\infty$ since $\mathbb{P}(T=\infty)>0$ and
$$
L_1(0)=  \mathbb{P}(T_1<T_2) ,
\quad
L_2(0)= \mathbb{P}(T_2<T_1),
\quad
L_1(0)+L_2(0)=\mathbb{P}(T<\infty).
$$
\begin{proposition}[Functional equation]
For $(p,q)\in\mathbb{C}^2$ such that $\Re p <0$ and $\Re q<0$ the Laplace transforms converge and we have
\begin{equation}
K(p,q)L(p,q)=L_1(q)+ L_2(p)-e^{px_0+qy_0}
\label{eq:eqfunc}
\end{equation}
where the kernel $K$ is defined by
\begin{equation}
K(p,q):=\frac{1}{2}(p-q)^2+(1-\gamma) p+\gamma q.
\label{eq:K}
\end{equation}
\label{prop:eqfunc}
\end{proposition}
\begin{proof}
We denote 
\begin{equation}
\mathcal{G}:=\frac{1}{2}\left(\frac{\partial}{\partial x}-\frac{\partial}{\partial y}\right)^2+(1-\gamma) \frac{\partial}{\partial x}+\gamma \frac{\partial}{\partial y}
\label{eq:Generateur}
\end{equation} 
the infinitesimal generator of $Z$. We apply Dynkin's formula, for $f$ twice differentiable we obtain
$$
\mathbb{E} f(Z(t\wedge T))=f(z_0)+ \mathbb{E} \int_0^{t \wedge T} \mathcal{G}f(Z(s))\mathrm{d}s 
$$
which can be rewritten
$$
\mathbb{E} (f(Z(t\wedge T_1))\mathds{1}_{T_1<T_2})+\mathbb{E} (f(Z(t\wedge T_2))\mathds{1}_{T_2<T_1})=f(z_0)+ \mathbb{E} \int_0^{t} \mathcal{G}f(Z(s)) \mathds{1}_{s<T}\mathrm{d}s .
$$
By taking $f(x,y)=e^{px+qy}$ we obtain 
$$
\mathbb{E}\left(e^{q Y({t\wedge T_1})} \mathds{1}_{T_1<T_2} \right)+\mathbb{E}\left(e^{p X({t\wedge T_2})} \mathds{1}_{T_2<T_1} \right)=e^{px_0+qy_0}+ K(p,q) \mathbb{E}\left(\int_0^t e^{(p,q)\cdot Z_s} \mathds{1}_{s<T} \mathrm{d}s \right) .
$$
The two expectations of the left-hand side are bounded by $1$ since $\Re p <0$ and $\Re q<0$ and then the expectation of the right-hand side is also finite.
By making $t$ tend towards infinity 
and by definition of $L$, $L_1$ and $L_2$ we obtain the functional equation~\eqref{eq:eqfunc}.
\end{proof}

\paragraph{Study of the kernel and analytic continuation}

The kernel $K(p,q)$ defined in~\eqref{eq:K}
leads us to introduce four algebraic functions $P^+$, $P^-$, $Q^+$ and $Q^-$ satisfying
$$
K(P^\pm(q),q)=K(p,Q^\pm(p))=0,
$$
analytic on $\mathbb{C}\setminus [(1-\gamma)^2/2,\infty)$ for $P^\pm$, on $\mathbb{C}\setminus [\gamma^2/2,\infty)$ for $Q^\pm$, and defined by 
\begin{equation}
P^\pm(q):=\gamma-1+q\pm\sqrt{(1-\gamma)^2-2q}
\quad\text{and}\quad
Q^\pm(p):=-\gamma+p\pm\sqrt{\gamma^2-2p}.
\label{eq:PQ}
\end{equation}
Thanks to the functional equation~\eqref{eq:eqfunc} we can now continue meromorphically $L_1$ to the domain $\{q\in\mathbb{C}:\Re q < (1-\gamma)^2 /2\}$ by the formula
$$
L_1(q)= e^{P^-(q) x_0 +q y_0} - L_2(P^-(q)).
$$
We just have to verify that $\Re P^-(q) <0$ when $\Re q < (1-\gamma)^2 /2$ and to remember that $L_2$ is analytic in the complex half-plane with negative real part.
In the same way, we extend $L_2$ analytically to the domain $\{p\in\mathbb{C}:\Re p < \gamma^2 /2\}$.

\paragraph{PDE and derivative of the Green's functions}

Let $z=(x,y)\in\mathbb{R}^2$.
Let us recall that we defined in~\eqref{eq:Generateur}
the infinitesimal generator $\mathcal{G}$. We now define its dual $\mathcal{G}^*$ by
$$\mathcal{G}^*:=\frac{1}{2}\left(\frac{\partial}{\partial x}-\frac{\partial}{\partial y}\right)^2-(1-\gamma) \frac{\partial}{\partial x}-\gamma \frac{\partial}{\partial y} .$$
The Green's function $g_{z_0}$ satisfies the following classical parabolic partial differential equation~\cite[Chap IX \S 17]{Doob2001}, 
denoting $\delta_{z_0}$ the Dirac measure in $z_0\in\mathbb{R}^2$ we have
\begin{equation}
\begin{cases}
\mathcal{G}^* g_{z_0}=-\delta_{z_0} & \text{on } (\mathbb{R}_+^*)^2,
\\
g_{z_0}=0 & \text{on } \mathbb{R}_+\times \{0\} \cup  \{0\} \times \mathbb{R}_+.
\end{cases}
\label{eq:pde}
\end{equation}

\begin{proposition}[Exit densities seen as derivative of the Green's function]
We have the following links between the Green's density and the exit densities:
$$
f_1(y)= \frac{1}{2}\frac{\partial g_{z_0}}{\partial x} (0,y) ,
\quad
f_2(x)=\frac{1}{2}\frac{\partial g_{z_0}}{\partial y} (x,0).
$$
\label{prop:greenderiv}
\end{proposition}

\begin{proof}
A weak formulation of the partial differential equation~\eqref{eq:pde} applied to an exponential test function leads to a functional equation linking the Laplace transforms of the Green's function and their derivatives. 
To do this, it is enough to integrate the partial differential equation~\eqref{eq:pde} and to make successive integration by parts on the integral 
$$\iint_{\mathbb{R}_+^2} \mathcal{G}^*g(z_0,z)e^{z\cdot(p,q)}\mathrm{d}z = -e^{z_0\cdot (p,q)}.$$
Then, defining $\widehat L_1 (q):=\frac{1}{2}\int_0^\infty \frac{\partial g}{\partial x} (0,y) e^{q y}\mathrm{d}y $ and symmetrically $\widehat L_2 (p)$ we obtain
$$
K(p,q)L(p,q)=\widehat L_1(q)+ \widehat L_2(p)-e^{px_0+qy_0}.
$$
Comparing this equation with the functional equation~\eqref{eq:eqfunc}, we obtain by identification that $\widehat L_1 =L_1$ and $\widehat L_2 =L_2$ and then we can conclude.
\end{proof}

\begin{remark}[Harmonic functions]
A function $u(t,y)$ satisfies~\eqref{eq:harmyt} if and only if $h(x,y):=u(x+y,y)$ satisfies $\mathcal{G}h=0$ on $(\mathbb{R}_+^*)^2$ and $h=0$ on the boundaries $\mathbb{R}_+\times \{0\} \cup  \{0\} \times \mathbb{R}_+$. We say that $h$ (resp. $u$) is harmonic for the killed degenerate Brownian motion $Z$ (resp. is harmonic for the killed space-time Brownian motion $B$).
\end{remark}

\section{Asymptotics via Malyshev's analytical approach}
\label{sec:martin}

\subsection{Asymptotics of Green's functions}

By using the functional equation~\eqref{eq:eqfunc}, inverting the Laplace transform and applying the steepest descent method, it is possible to compute the asymptotics of Green's function.

\begin{theorem}[Asymptotics of Green's function and exit densities]
Let $\alpha\in (0,\pi/2)$ and $z=(x,y)\in\mathbb{R}_+^2$. When $z\to\infty$ and $y/x \to \tan\alpha$ we have
\begin{equation}
g_{z_0}(z)\sim  \frac{h^\alpha(z_0)}{\sqrt{|z|}} \frac{e^{-z\cdot (p(\alpha),q(\alpha))}}{\sqrt{2\pi (\cos\alpha+\sin\alpha)}}
\label{eq:asymptmain}
\end{equation}
where the so-called saddle point $(p(\alpha),q(\alpha))$ is defined by
\begin{align}
(p(\alpha),q(\alpha))&:=\mathrm{argmax} \{p\cos\alpha+q\sin\alpha :(p,q)\in\mathbb{R}^2, K(p,q)=0 \}
\label{eq:saddlepoint}
\\
&= \left( \frac{\gamma^2 }{2} - \frac{1}{2(1+\frac{1}{\tan \alpha})^2} , \frac{(1-\gamma)^2}{ 2} - \frac{1}{2(1+{\tan \alpha})^2} \right)
\label{eq:saddlepoint2}
\end{align}
and $h^\alpha$ is a harmonic function for the process (i..e $\mathcal{G} h^\alpha=0 $ on $\mathbb{R}_+^2$ and $h^\alpha=0$ on $\mathbb{R}\times \{0\} \cup  \{0\} \times \mathbb{R}$) defined for $\alpha\in(0,\pi/2)$ by
\begin{equation}
h^\alpha(z_0):= e^{p(\alpha) x_0+q(\alpha) y_0}-L_1(q(\alpha))-L_2(p(\alpha)) .
\label{eq:halpha}
\end{equation}
We also have the asymptotics of the exit densities on the boundaries
\begin{equation}
f_1(y)\underset{y\to\infty}{\sim} \frac{h^{\pi/2}(z_0)}{y^{3/2}}\frac{e^{-y q(\pi/2)}}{\sqrt{2\pi}}
\quad \text{and} \quad
f_2(x)\underset{x\to\infty}{\sim} \frac{h^0(z_0)}{x^{3/2}}\frac{e^{-x p(0)}}{\sqrt{2\pi}}
\label{eq:asymptboundaries}
\end{equation}
where we define
$$
h^{\pi/2}(z_0):=x_0 e^{x_0 p(\pi/2) +y_0 q(\pi/2)}-L_2'(p(\pi/2))
\quad \text{and} \quad
h^0(z_0):=y_0 e^{x_0 p(0) +y_0 q(0)} - L_1'(q(0))
.
$$
\label{prop:asympt}
\end{theorem}
The proof is done below in Section~\ref{subsec:proofthm}.

\begin{remark}[$h^0$ and $h^{\pi/2}$ as derivatives]
It should be noted that for $\alpha>0$ we have $h^\alpha \to 0$ when $\alpha\to 0$ (it derives from Proposition~\ref{prop:conditionedescapeproba} below). That is why we have defined $h^0$ differently. We can interpret $h^0$ 
as the $q$-derivative of $h^\alpha$ when $\alpha>0$ evaluated in $q(0)$,
i.e. $h^0= \left. \frac{\partial h^\alpha}{\partial q} \right|_{q(0)}$. 
A similar remark holds for $h^{\pi /2}$.
\label{rem:qder}
\end{remark}

\begin{figure}[hbtp]
\centering
\includegraphics[scale=0.4]{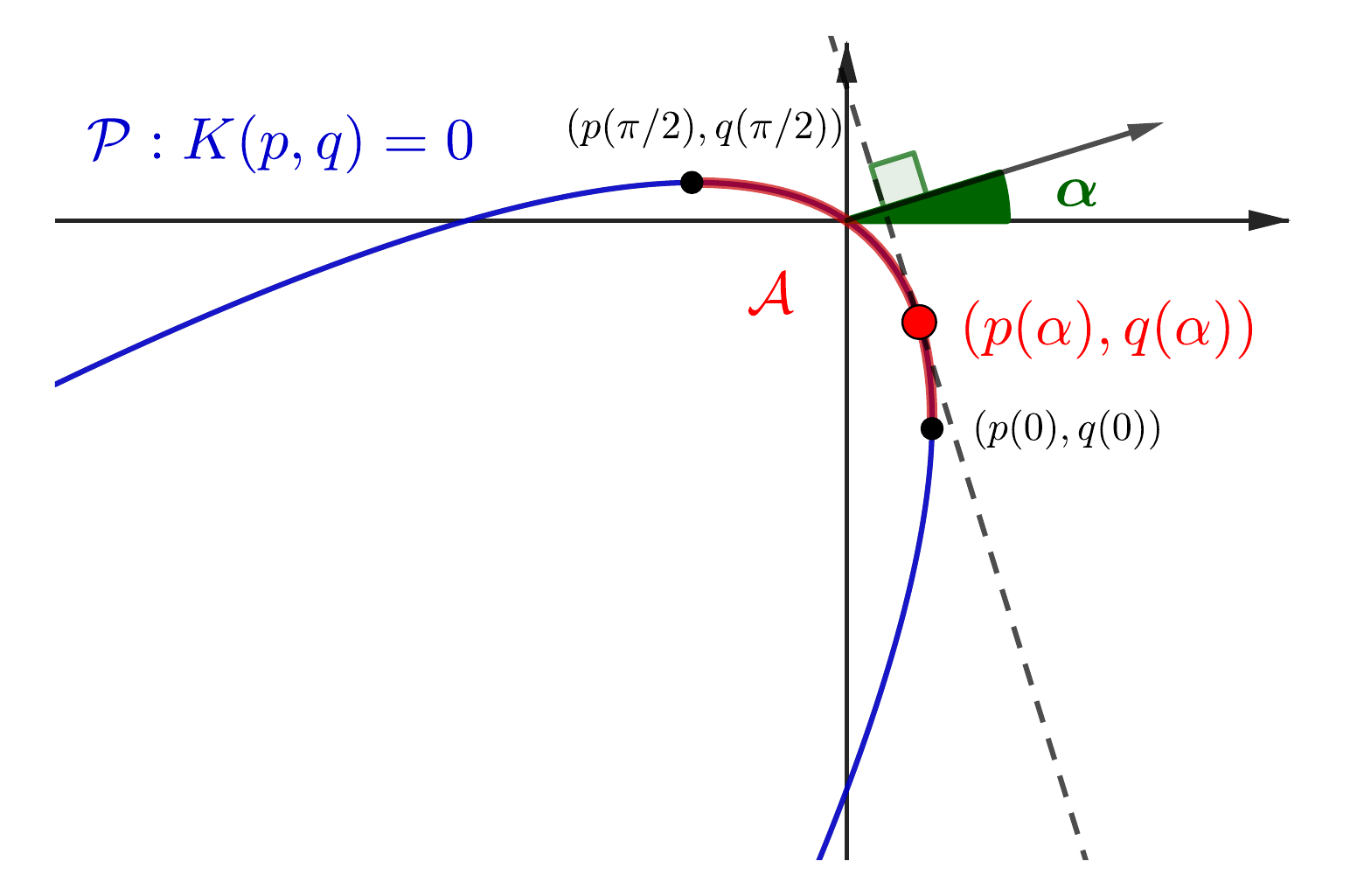}
\caption{The parabola $\mathcal{P}:=\{(p,q)\in\mathbb{R}_+^2 : K(p,q)=0 \}$ in blue and the arc of parabola $\overline{\mathcal{A}}:=\{(p(\alpha),q(\alpha)): \alpha\in[0,\pi/2] \} \subset \mathcal{P}$ which parametrizes the Martin's boundary in red   
}
\label{fig:pointcol}
\end{figure}

\subsection{Proof of Theorem~\ref{prop:asympt}}
\label{subsec:proofthm}

The proof of  Theorem~\ref{prop:asympt} follows the same steps as in \cite{ernst_franceschi_asymptotic_2021} and \cite{franceschi_kourkova_petit_2023}. We will first show asymptotics for the Green functions $g_{z_0}$~\eqref{eq:asymptmain}  and then asymptotics for the exit densities $f_1$ and $f_2$~\eqref{eq:asymptboundaries}.

\paragraph{From double to simple integral via the residue theorem}
The Laplace inversion formula, see \cite[Theorem 24.3 and 24.4]{doetsch_introduction_1974} and
\cite{brychkov_multidimensional_1992}, implies that for $\epsilon>0$ we have 

$$
g_{z_0}(z)= \frac{1}{(2i\pi)^2}
 \int_{-\epsilon-i\infty}^{-\epsilon+i\infty} \int_{-\epsilon-i\infty}^{-\epsilon+i\infty} L(p,q) e^{-px-qy} \mathrm{d}p\mathrm{d}q.
$$
We are now going to use the functional equation~\eqref{eq:eqfunc} and the residue theorem to transform the double integral into a sum of simple integrals.
\begin{lemma}[Reduction to simple integrals]
The Green's function satisfies $g_{z_0}(z)
= I_1+I_2+I_3$ where
\begin{equation}
\begin{cases}
I_1
:=
\frac{-1}{2i\pi}
\int_{-\epsilon-i\infty}^{-\epsilon+i\infty}
\frac{L_1(q) }{\sqrt{(1-\gamma)^2-2q}} e^{-P^+(q)x-qy}\mathrm{d}q,
\\
I_2:=
\frac{-1}{2i\pi}
\int_{-\epsilon-i\infty}^{-\epsilon+i\infty}
\frac{L_2(p) }{\sqrt{\gamma^2 -2p}} e^{-px-Q^+(p)y} \mathrm{d}p,
\\
I_3:=
\frac{1}{2i\pi}
\int_{-\epsilon-i\infty}^{-\epsilon+i\infty}
\frac{e^{p x_0+Q^+(p)y_0} }{\sqrt{\gamma^2 -2p}} e^{-px-Q^+(p)y}\mathrm{d}p.
\end{cases}
\label{eq:I1I2I3}
\end{equation}
\label{lem:I123}
\end{lemma}
\begin{proof}
The following equalities are explained below.
\begin{align*}
g(z_0,z) &= \frac{1}{(2i\pi)^2}\int_{-\epsilon-i\infty}^{-\epsilon+i\infty} \int_{-\epsilon-i\infty}^{-\epsilon+i\infty}
\frac{L_1(q)+L_2(p)-e^{px_0+qy_0}}{K(p,q)}
e^{-px-qy} \mathrm{d}p\mathrm{d}q
\\ &=
\frac{1}{2i\pi}
\int_{-\epsilon-i\infty}^{-\epsilon+i\infty}
\frac{L_1(q)  }{\frac{1}{2}(P^+(q)-P^-(q))} e^{-P^+(q)x-qy}\mathrm{d}q
+
\frac{1}{2i\pi}
\int_{-\epsilon-i\infty}^{-\epsilon+i\infty}
\frac{L_2(p) }{\frac{1}{2}(Q^+(p)-Q^-(p))} e^{-px-Q^+(p)y} \mathrm{d}p
\\ &-
\frac{1}{2i\pi} \int_{-\epsilon-i\infty}^{-\epsilon+i\infty}
\frac{-e^{p x_0+Q^+(p)y_0} }{\frac{1}{2}(Q^+(p)-Q^-(p))} e^{-px-Q^+(p)y}\mathrm{d}p
\\ &=I_1+I_2+I_3
\end{align*}
The first equality above comes from the functional equation~\eqref{eq:eqfunc}. The second equality above comes from the residue theorem applied to an integration contour $\mathcal{C}_R$ represented in Figure~\ref{fig:intcont} to obtain $I_1$ (and a similar contour to obtain $I_2$ and $I_3$). In a classic way, the residue theorem implies that
$\frac{1}{2i\pi}\int_{\mathcal{C}_R}
\frac{L_1(q)}{K(p,q)}
e^{-px-qy} \mathrm{d}p = \frac{L_1(q)}{\frac{1}{2}(P^+(q)-P^-(q))} e^{-P^+(q)x-qy} $ since $P^+(q)$ is the only pole of the integrand inside $\mathcal{C}_R$.
Furthermore, the integral on the red half circle and the blue segments tends to $0$ when $R \to \infty$ and we obtain that $-\int_{-\epsilon-i\infty}^{-\epsilon+i\infty}
\frac{L_1(q)}{K(p,q)}
e^{-px-qy} \mathrm{d}p =\lim_{R\to\infty} \int_{\mathcal{C}_R}
\frac{L_1(q)}{K(p,q)}
e^{-px-qy} \mathrm{d}p$ taking into account the orientation of the contour. Details are omitted and one can refer to \cite[Lemma 4.1]{franceschi_kourkova_petit_2023} or \cite[Proposition 5.]{ernst_franceschi_asymptotic_2021} which perform similar calculations. 
The third equality just comes from the definition of $P^\pm$ and $Q^\pm$ in~\eqref{eq:PQ}.
\begin{figure}[hbtp]
\centering
\includegraphics[scale=0.13]{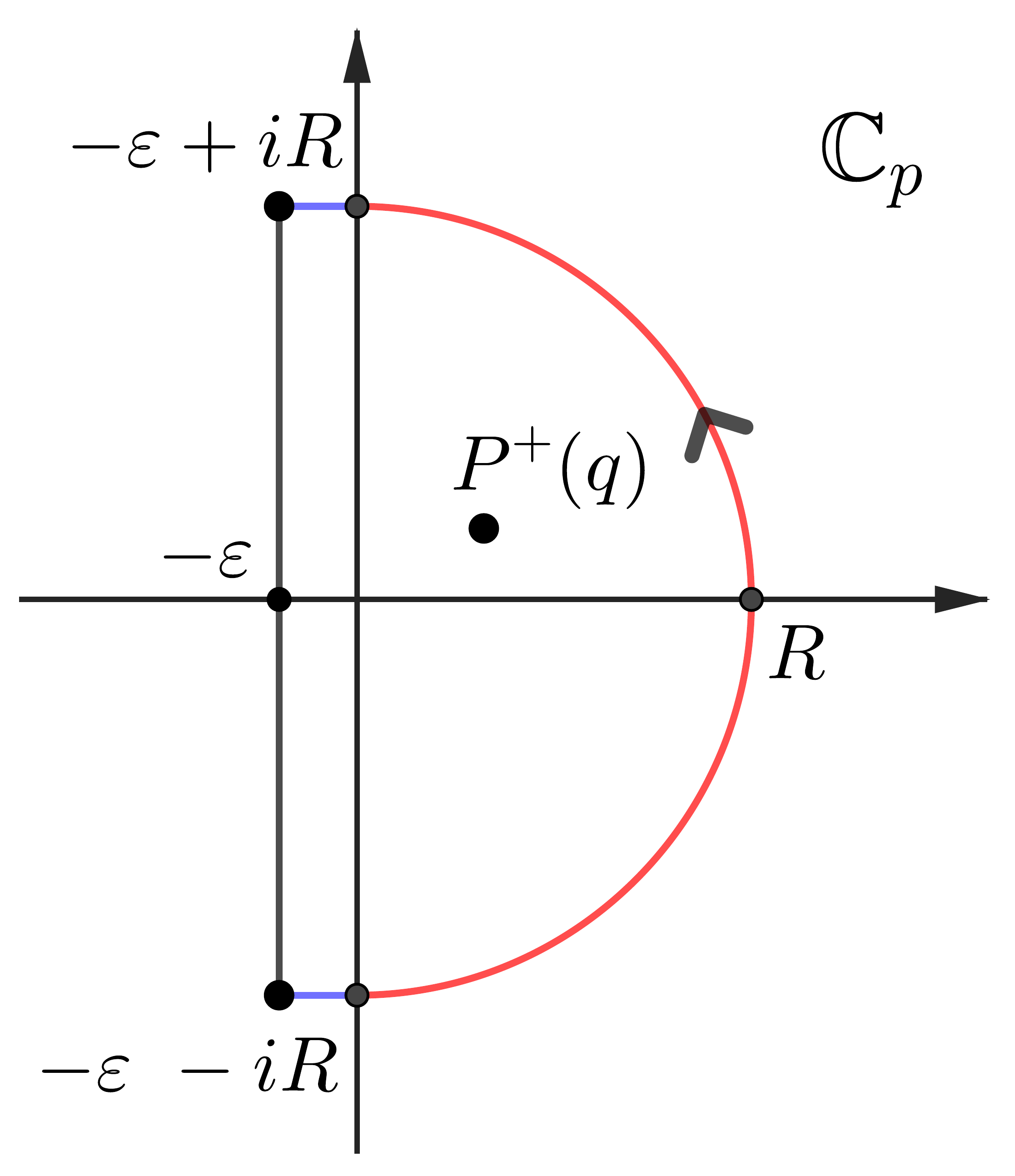}
\caption{Integration contour $\mathcal{C}_R$ and $Q^+(p)$ the only pole of the integrand inside $\mathcal{C}_R$}
\label{fig:intcont}
\end{figure}

\end{proof}

\paragraph{Steepest descent method}
We are going to denote  by $(r,\alpha)\in\mathbb{R}_+\times [0,\pi/2]$ the polar coordinate of the point $z=(x,y)\in \mathbb{R}_+^2$, i.e. $(x,y)=(r\cos\alpha,r \sin \alpha)$ and $(r,\alpha)=(\sqrt{x^2+y^2}, \arctan (y/x))$.
Let us define the functions
$$
F(p,\alpha):= p \cos \alpha +Q^+(p) \sin \alpha
\quad
\text{and}
\quad
G(q,\alpha):= P^+(q) \cos \alpha +q \sin \alpha.
$$
The point $(p(\alpha),q(\alpha))$ defined in~\eqref{eq:saddlepoint} is a saddle point for these functions. Indeed, 
for $\alpha\in (0,\pi/2)$ 
the first derivatives (according to $p$ for $F$ and $q$ for $G$) satisfy
$$
F'_p(p(\alpha),\alpha)=G'_q(q(\alpha),\alpha)=0,
$$
and the second derivatives satisfy
\begin{equation}
F''_p(p(\alpha),\alpha)
=-\frac{(\cos\alpha+\sin\alpha)^3}{\sin^2\alpha}
< 0, \quad 
G''_q(q(\alpha),\alpha)=-\frac{(\cos\alpha+\sin\alpha)^3}{\cos^2\alpha}< 0.
\label{eq:F''}
\end{equation}

\begin{lemma}[Asymptotics of $I_1$] 
When $r\to\infty$ and $\alpha\to\alpha_0$ the following asymptotics holds:
$$
I_1  = - L_1(q(\alpha_0))
\frac{e^{-z\cdot (p(\alpha_0),q(\alpha_0))}}{\sqrt{2\pi r (\cos\alpha_0+\sin\alpha_0) }}
+o\left( \frac{1}{\sqrt{r}} e^{-z\cdot (p(\alpha_0),q(\alpha_0))} \right) .
$$
Similar asymptotics hold for $I_2$ and $I_3$. 
\label{lem:asymptI1}
\end{lemma}
\begin{proof}
The integral $I_1$ is typical to apply the classical saddle point method. A famous reference is Fedoryuk's book
\cite[\S 4]{fedoryuk_asymptotic_1989}.
Starting from~\eqref{eq:I1I2I3} we then obtain
\begin{align*}
I_1
&=\frac{-1}{2i\pi} \int_{-\epsilon-i\infty}^{-\epsilon+i\infty}
\frac{L_1(q) }{\sqrt{(1-\gamma)^2-2q}} e^{-rG(q,\alpha)}\mathrm{d}q
\\
&
\underset{r\to\infty \atop \alpha\to\alpha_0}{=} \frac{-1}{2i\pi} \frac{L_1(q(\alpha_0)) }{\sqrt{(1-\gamma)^2-2q(\alpha_0)}}
i\sqrt{\frac{-2\pi}{r G''_q (q(\alpha_0))}}
e^{-r G(q(\alpha_0),\alpha_0)}
+o\left( \frac{1}{\sqrt{r}} e^{-r G(q(\alpha_0),\alpha_0)} \right) .
\end{align*}
Let us give some details about the application of this method.
To obtain the asymptotics we shift the integration contour onto the steepest descent line which passes through the saddle point $q(\alpha)$. This is possible since the integrand is a meromorphic function and we do not cross the only pole since $q(\alpha)<(1-\gamma)^2/2$ for $\alpha\in(0,\pi/2)$, see~\eqref{eq:saddlepoint}. Then, the asymptotics derives from the classical Laplace's method, see for example~\cite[Lemma 16]{ernst_franceschi_asymptotic_2021}. To handle the convergence in $\alpha\to\alpha_0$ we just apply a simple parameter-dependent Morse lemma, see \cite[Appendix A]{franceschi_kourkova_petit_2023} which details this technique.
To conclude, a straightforward calculation using~\eqref{eq:saddlepoint2} and~\eqref{eq:F''} shows that
$$
\frac{1}{2\pi} \frac{1 }{\sqrt{(1-\gamma)^2-2q(\alpha)}}
\sqrt{\frac{-2\pi}{ G''_q (q(\alpha))}}
=
\frac{1}{\sqrt{2\pi  (\cos\alpha+\sin\alpha) }}
$$ 
and it just remains to remark that $r G(q(\alpha),\alpha)=z\cdot (p(\alpha),q(\alpha))$.
\end{proof}

\begin{proof}[Proof of~\eqref{eq:asymptmain}]
The main asymptotics~\eqref{eq:asymptmain} of  Theorem~\ref{prop:asympt} directly derives from Lemmas~\ref{lem:I123} and~\ref{lem:asymptI1}. One will compute more explicitly $h^\alpha(z_0)=e^{p(\alpha) x_0+q(\alpha) y_0}-L_1(q(\alpha))-L_2(p(\alpha))$ in the following section which will show that this function is non-zero for $\alpha\in(0,\pi/2)$.
\end{proof}


\paragraph{Asymptotics on the boundaries via Tauberian lemmas}

\begin{lemma}[Branching point behaviour]
When $q\to q^+ := q(\pi/2)=(1-\gamma)^2/2 $ we have
$$
L_1(q)=L_1(q^+)-
(x_0 e^{P^-(q^+) x_0 +q^+ y_0} -L_2' (P^-(q^+))) \sqrt{2(q^+-q)} +  o(\sqrt{q^+-q})
$$
and a symmetrical result holds for $L_2$.
\label{lem:branch}
\end{lemma}
\begin{proof}
Evaluating 
\eqref{eq:eqfunc} at $(P^-(q),q)$ and making a Taylor series expansion when $q\to q^+ $ we obtain 
\begin{align*}
L_1(q) &= e^{P^-(q) x_0 +q y_0} -L_2(P^-(q))
\\ &= 
L_1(q^+)+
(x_0 e^{P^-(q^+) x_0 +q^+ y_0} -L_2' (P^-(q^+))) (P^-(q)) - P^-(q^+)) +  o(P^-(q)) - P^-(q^+))
\end{align*}
and it just remains to notice that
$P^-(q)) - P^-(q^+)= - \sqrt{2(q^+-q)}$.
\end{proof}

\begin{proof}[Proof of~\eqref{eq:asymptboundaries}] Lemma~\ref{lem:branch} and classical Tauberian inversion 
lemmas
\cite[Theorem 37.1]{doetsch_introduction_1974},
\cite[Lemma C.2]{dai_reflecting_2011},
imply that

$$
f_1(y)\underset{y\to\infty}{\sim} (x_0 e^{P^-(q^+) x_0 +q^+ y_0} -L_2' (P^-(q^+)))\frac{-\sqrt{2}}{\Gamma(-1/2)} y^{-3/2}e^{-q^+ y}
$$
and we deduce the asymptotics~\eqref{eq:asymptboundaries} of Theorem~\ref{prop:asympt} noticing that $\Gamma(-1/2)=-2\sqrt{\pi}$.
\end{proof}

\section{Martin boundary and persistence probabilities}
\label{sec:mart}

\subsection{Doob's $h$-transform and persistence}

In this short section, we give a probabilist interpretation of the harmonic functions $h^\alpha$, found in Theorem~\ref{prop:asympt}, in term of the persistence probabilities of the process conditioned to drift in a given direction.

For $\alpha\in[0,\pi/2]$ we consider the probability $\mathbb{P}^\alpha$ defined by the Doob's $h$-transform for the harmonic function $e^{xp(\alpha)+y q(\alpha)}$ associated to the process $Z(t)=z_0+\mathbf{W}(t)+\mu t$.
We have
$$
\mathbb{P}^\alpha (Z(t) \in S)=\mathbb{E} ( e^{(p(\alpha),q(\alpha)) \cdot (\mathbf{W}(t)+\mu t)} \mathds{1}_{Z(t) \in S})
$$
for a set $S\subset \mathbb{R}^2$. 
Under $\mathbb{P}^\alpha$ the process $Z(t)$ is a Markov process of transition kernel 
$$\widetilde P_t(z_0,z)=e^{(x-x_0)p(\alpha)+(y-y_0) q(\alpha)}P_t(z_0,z)$$
where $P_t$ is the transition kernel of $Z(t)$ under $\mathbb{P}$. Under $\mathbb{P}^\alpha$ the process $Z(t)$ is conditioned to drift in the direction $\alpha$ (in the sense of the Doob's $h$-transform).

We introduce the corresponding modified Green's measure, for $S\subset \mathbb{R}_+^2$ we set
$$
G^\alpha_{z_0}(S):=\int_0^\infty \mathbb{P}^\alpha(Z(t)\in S ,t<T)\mathrm{d}t
$$
and its density
$$
g^\alpha_{z_0} (z):=e^{(x-x_0)p(\alpha)+(y-y_0) q(\alpha)} g_{z_0}(z).
$$
We define $L^\alpha$ the Laplace transform relative to these modified Green's functions and we have
$$
L^\alpha (p,q)=L(p+p(\alpha),q+q(\alpha)) e^{-x_0 p(\alpha)-y_0 q(\alpha)}.
$$
On the same way, we introduce the modified exit measures
$$
A_1^\alpha(S ) :=\mathbb{P}^\alpha(Z(T_1)\in S , T_1<T_2)
\quad\text{and}\quad
A_2^\alpha(S) :=\mathbb{P}^\alpha(Z(T_2)\in S , T_2<T_1)$$
and their densities 
$$f_1^\alpha(y):=e^{-x_0 p(\alpha)+(y-y_0) q(\alpha)} f_1(y)
\quad\text{and}\quad
f_2^\alpha (x):= e^{(x-x_0)p(\alpha)-y_0 q(\alpha)} f_2(x).$$
We denote their Laplace transforms $L_1^\alpha$ and $L_2^\alpha$ and we have
$$
L_1^\alpha (q)=L_1(q+q(\alpha)) e^{-x_0 p(\alpha)-y_0 q(\alpha)},
\quad
L_2^\alpha (p)=L_2(p+p(\alpha)) e^{-x_0 p(\alpha)-y_0 q(\alpha)}
.
$$
One may remark that
$$
L_1^\alpha(0)=  \mathbb{P}^\alpha(T_1<T_2) ,
\quad
L_2^\alpha(0)= \mathbb{P}^\alpha(T_2<T_1),
\quad
L_1^\alpha(0)+L_2^\alpha(0)=\mathbb{P}^\alpha(T<\infty).
$$

\begin{proposition}[Conditioned persistence probabilities]
Let $z_0=(x_0,y_0)$, for $\alpha\in (0,\pi)$ we have
\begin{equation}
h^{\alpha} (z_0)= e^{x_0 p(\alpha)+y_0 q(\alpha)} \mathbb{P}^\alpha(T=\infty)
=e^{x_0 p(\alpha)+y_0 q(\alpha)} - \mathbb{E} (e^{X_T p(\alpha)+Y_T q(\alpha)} \mathds{1}_{T<\infty}) 
\label{eq:escapeprobah}  
\end{equation}
and 
$$
h^0(z_0)=y_0e^{x_0 p(0) +y_0 q(0)}- \mathbb{E}\left(Y_{T_1} e^{q(0) Y_{T_1}} \mathds{1}_{T_1<T_2} \right),
\quad
h^{\pi/2}(z_0)=x_0e^{x_0 p(\pi/2) +y_0 q(\pi/2)}- \mathbb{E}\left(X_{T_2} e^{p(\pi/2) X_{T_2}} \mathds{1}_{T_2<T_1} \right)
.
$$
When $x_0\to\infty $ and $y_0\to\infty$
we obtain
\begin{equation}
h^\alpha(z_0) \sim e^{x_0 p(\alpha)+y_0 q(\alpha)},
\quad
h^0(z_0) \sim y_0e^{x_0 p(0) +y_0 q(0)} 
\quad \text{and} \quad
h^{\pi/2}(z_0) \sim x_0 e^{x_0 p(\pi/2)} .
\label{eq:asymptz0}
\end{equation}
\label{prop:conditionedescapeproba}
\end{proposition}
\begin{proof}
By Theorem~\ref{prop:asympt}, we have
\begin{align*}
h^\alpha(z_0)&= e^{p(\alpha) x_0+q(\alpha) y_0}-L_1(q(\alpha))-L_2(p(\alpha))
\\
&= e^{p(\alpha) x_0+q(\alpha)y_0}(1-L_1^\alpha(0)-L_2^\alpha(0))
\\
&= e^{p(\alpha) x_0+q(\alpha)y_0}(1-\mathbb{P}^\alpha(T_1<T_2)-\mathbb{P}^\alpha(T_2<T_1))
\\
&= e^{p(\alpha) x_0+q(\alpha)y_0}\mathbb{P}^\alpha (T=\infty) 
\\
&= e^{p(\alpha) x_0+q(\alpha)y_0}-e^{p(\alpha) x_0+q(\alpha)y_0}\mathbb{P}^\alpha (T<\infty) 
\\ &=
e^{x_0 p(\alpha)+y_0 q(\alpha)} - \mathbb{E} (e^{X_T p(\alpha)+Y_T q(\alpha)} \mathds{1}_{T<\infty}) 
,
\end{align*}
and also
\begin{align*}
h^0(z_0) &=y_0 e^{x_0 p(0) +y_0 q(0)}-  L_1'(q(0))
\\ &=
y_0 e^{x_0 p(0) +y_0 q(0)}- 
\mathbb{E}\left(Y_{T_1} e^{q(0) Y_{T_1}} \mathds{1}_{T_1<T_2} \right) .
\end{align*}
We now assume that $x_0\to\infty $ and $y_0\to\infty$. Noticing that $\mathbb{P}^\alpha(T=\infty) {\to} 1$ we obtain the asymptotics for $h^\alpha$ when $\alpha\in(0,\pi/2)$. Remembering that $p(0)>0$, $q(0)<0$ and that $Y_{T_1}>y_0$, we see that $\mathbb{E}\left(Y_{T_1} e^{q(0) Y_{T_1}} \mathds{1}_{T_1<T_2} \right)=o(y_0 e^{x_0 p(0) +y_0 q(0)})$ and the asymptotics of $h^0$ follows. A symmetrical proof holds for $h^{\pi/2}$.
\end{proof}
These asymptotics will be useful in the proof of Theorem~\ref{prop:minimalharmonic} to show that the Martin boundary is minimal. 

\subsection{Martin boundary}

We refer to the classic book by Doob~\cite{Doob2001} which gives a comprehensive presentation of Martin boundary theory.
\begin{theorem}[Martin boundary]
For $\alpha\in[0,\pi/2]$, the function $z_0\mapsto h^\alpha(z_0)$ is harmonic for the process $Z$, i.e. $\mathcal{G} h^\alpha =0 $ inside the quarter plane and $h^\alpha =0$ on its boundaries. The Martin boundary of the process $Z$ killed outside of the quarter plane is denoted $\partial_M^{\mathbb{R}_+^2} Z$ and is homeomorphic to $[0,\pi/2]$ through the map $\Phi$ given by
\renewcommand{\arraystretch}{1.7}
$$
\begin{array}{ccccc}
 \Phi : [0,\frac{\pi}{2}] & \longrightarrow & \partial_M^{\mathbb{R}_+^2} Z \\
 \alpha & \longmapsto & \displaystyle \frac{h^\alpha (\cdot)}{h^\alpha (1,1)} . \\
\end{array}
$$
\label{cor:martinboundary}
\end{theorem}
\begin{proof}
We define the Martin's kernel for the arbitrary point $(1,1)\in\left( \mathbb{R}_+^*\right)^2$ by
$$
k(z_0,z):=\frac{g_{z_0}(z)}{g_{(1,1)}(z)}.
$$
By the Martin's boundary theory in a parabolic context~\cite[Chap XIX]{Doob2001} we have to show that for $\alpha\in [0,\pi/2]$, $z=(x,y)$, $z\to\infty$ and $y/x \to \tan\alpha$ we have
\begin{equation*}
k(z_0,z)\rightarrow k^\alpha(z_0):= \frac{h^\alpha (z_0)}{h^\alpha (1,1)}.
\end{equation*}
For $\alpha\in(0,\pi/2)$ it directly derives from the asymptotics of the Green's function $g_{z_0}$ \eqref{eq:asymptmain} found in Theorem~\ref{prop:asympt}. For $\alpha=0$ or $\pi/2$ it is enough to use formula \eqref{eq:asymptboundaries} of Theorem~\ref{prop:asympt} and Proposition~\ref{prop:greenderiv} and to apply l'Hôpital's rule. The continuity of $\Phi$ is clear on $(0,\pi/2)$. The continuity at $0$ derives from Remark~\ref{rem:qder}, noticing that $p'(0)=0$, with a simple Taylor approximation we obtain that $h^\alpha=\alpha q'(0)h^0 +o(\alpha) $ when $\alpha\to 0$ and then $k^\alpha \to k^0$ when $\alpha\to 0$. The same reasoning holds for the continuity at $\pi/2$.
The injectivity of $\Phi$ follows from the asymptotics of $h^\alpha$ when $x_0\to\infty$ and $y_0\to\infty$ which are all different for different values of $\alpha$, see~\eqref{eq:asymptz0} of Proposition~\ref{prop:conditionedescapeproba}.
Since $[0,\pi/2]$ is compact and $\Phi$ is a continuous bijection, $\Phi$ is an homeomorphism.
\end{proof}
In Theorem~\ref{prop:minimalharmonic}, we will deduce from the asymptotics~\eqref{eq:asymptz0} that the Martin boundary is minimal. 

\section{Harmonic functions via the compensation approach}
\label{sec:harmcompensation}

\subsection{Minimal harmonic functions}

The harmonic function $h^\alpha(z_0)$ defined in~\eqref{eq:halpha} has been expressed in terms of the Laplace transforms $L_1$ and $L_2$ and then interpreted in~\eqref{eq:escapeprobah} as a persistence probability of a conditioned process. But this is not a very tractable expression since these Laplace transforms and these probabilities are not known explicitly. 
In this section, we compute explicitly all the harmonic functions $h^\alpha(z_0)$ thanks to a recursive compensation approach.  

\begin{proposition}[Harmonic functions via compensation approach, $\alpha\in(0,\pi/2)$]
For $z=(x,y)\in\mathbb{R}_+^2$, we define the functions $\widetilde h^{(p_0,q_0)}$ by
\begin{equation}
\widetilde h^{(p_0,q_0)}(x,y)
:=
\sum_{n\in\mathbb{Z}} (-1)^n e^{x p_n +y q_n}
\label{eq:defhpq}
\end{equation}
where $(p_n,q_n)$ is a recursive sequence of points on the parabola $\mathcal{P}:=\{(p,q)\in\mathbb{R}_+^2 : K(p,q)=0 \}$ starting from $(p_0,q_0)\in\mathcal{A}:=\{(p(\alpha),q(\alpha)): \alpha\in(0,\pi/2) \}$ the red arc of the parabola, see Figures~\ref{fig:pointcol} and~\ref{fig:seq}, and defined by
\begin{equation}
\begin{cases}
p_{2n}:=p_0+2n(p_0-q_0)-2n(n+\gamma)
\\
q_{2n}:=q_0+2n(p_0-q_0)-2n(n+\gamma-1)
\end{cases}
\quad\text{and}\quad
(p_{2n+1},q_{2n+1})=(p_{2n},q_{2n+2}).
\label{eq:pnqn}
\end{equation} 
Then, the functions $\widetilde h^{(p_0,q_0)}$ are positive and harmonic for the process $Z$ killed at the boundary of the quadrant, that is
\begin{equation}
\begin{cases}
\mathcal{G} \widetilde h^{(p_0,q_0)}  =0 
& \text{on } (\mathbb{R}_+^*)^2
\\
\widetilde h^{(p_0,q_0)} =0 & \text{on } \mathbb{R}\times \{0\} \cup  \{0\} \times \mathbb{R}.
\end{cases}
\label{eq:harm}
\end{equation}
\label{prop:compensationharm}
\end{proposition}
\begin{proof}
This sum is typical of the compensation approach. First, we may notice that the convergence of the sum is trivial since the sequence $(p_n,q_n)$ tends to minus infinity quadratically. Let us notice that a function $f(z)=e^{z\cdot (p,q)}$ satisfies $\mathcal{G}f=0$ if and only if $K(p,q)=0$, i.e. $(p,q)\in\mathcal{P}$. It is easy to verify that for all $n\in\mathbb{Z}$ we have $(p_n,q_n)\in\mathcal{P}$, see Figure~\ref{fig:seq}. Therefore, we deduce that $\mathcal{G} \widetilde h^{(p_0,q_0)}  =0$ for $z=(x,y)\in\mathbb{R}_+^2$. Furthermore, using $(p_{2n+1},q_{2n+1})=(p_{2n},q_{2n+2})$ we have the following recursive compensations
$$
\widetilde h^{(p_0,q_0)}(x,y) =  \overunderbraces{& &\br{4}{=0  \text{ when } x=0 }& \br{4}{=0  \text{ when } x=0 }}%
  {\cdots + & e^{p_{-2}x+q_{-2}y} &-& e^{p_{-1}x+q_{-1}y}  &+&e^{p_{0}x+q_{0}y} &-&e^{p_{1}x+q_{1}y} &+ & e^{p_{2}x+q_{2}y} &+ \cdots}%
  {& \br{3}{=0  \text{ when } y=0 } & &\br{3}{=0  \text{ when } y=0 }}
$$
and we deduce that $\widetilde h^{(p_0,q_0)}$ satisfies the boundary conditions of~\eqref{eq:harm}. 

It only remains to show the positivity. First, there exists $t>0$ large enough such that when $x+y>t$ we have $\widetilde h^{(p_0,q_0)}(z) >0$, indeed
\begin{align*}
\widetilde h^{(p_0,q_0)}(x,y) =  
\overbrace{ e^{p_{0}x+q_{0}y} -e^{p_{1}x+q_{1}y} - e^{p_{-1}x+q_{-1}y}  }^{\geqslant 0  \text{ when } x+y \text{ large enough}  }     %
 &+ \sum_{\mathbb{N}^*}
  \overbrace{e^{p_{2n}x+q_{2n}y} - e^{p_{2n+1}x+q_{2n+1}y}  }^{\geqslant 0 \text{ always}}
  \\ 
 &+  \sum_{\mathbb{N}^*} 
 \overbrace{e^{p_{-2n}x+q_{-2n}y} - e^{p_{-2n-1}x+q_{-2n-1}y} .}^{\geqslant 0 \text{ always}}
\end{align*}
For all starting point $z_0=(x_0,y_0)\in\mathbb{R}_+^2$ of the process $Z=(X,Y)$, remembering that $X(t)+Y(t)=t_0+t$ and since $\widetilde h^{(p_0,q_0)}$ is harmonic, we deduce that
$$
\widetilde h^{(p_0,q_0)} (x_0,y_0)=\mathbb{E}[\widetilde h^{(p_0,q_0)}(Z(t))]
\geqslant 0 .
$$
\end{proof}

\begin{figure}[hbtp]
\centering
\includegraphics[scale=0.67]{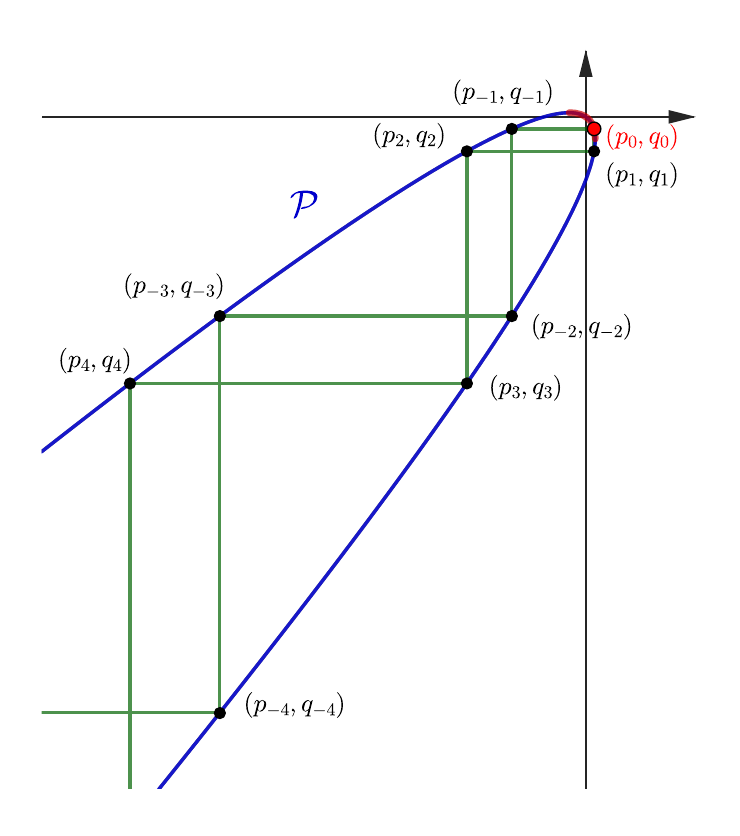}
\includegraphics[scale=0.67]{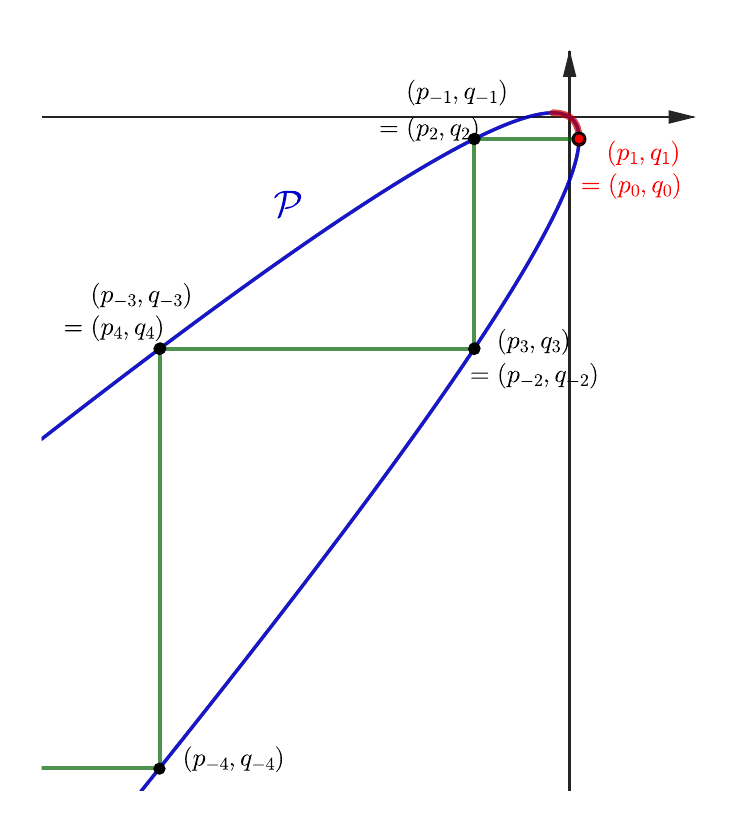}
\caption{Sequel of the points $(p_n,q_n)$ on the parabola $\mathcal{P}:=\{(p,q)\in\mathbb{R}^2 : K(p,q)=0\}$. On the left $(p_0,q_0)\in \mathcal{A}:=\{(p(\alpha),q(\alpha)): \alpha\in(0,\pi/2) \}$, on the right, $(p_0,q_0)=(p(0),q(0)) \in \partial \mathcal{A}$.}
\label{fig:seq}
\end{figure}

It remains to define the two harmonic functions associated with the extremal points of the arc $\partial{\mathcal{A}}= \{(p(0),q(0)),(p(\pi/2),q(\pi/2)) \}$. Indeed, when $(p_0,q_0)\to(p(0),q(0))$ or 
$(p(\pi/2),q(\pi/2))$ we have $\widetilde h^{(p_0,q_0)}\to 0$. That is why we have to define $\widetilde h^{(p(0),q(0))}$ and $\widetilde h^{(p(\pi/2),q(\pi/2))}$ differently. 
\begin{proposition}[Harmonic functions via compensation approach, $\alpha=0$ or $\pi/2$]
For $z=(x,y)\in\mathbb{R}_+^2$, we define the function $\widetilde h^{(p(0),q(0))}$ by
\begin{equation}
\widetilde h^{(p(0),q(0))}(x,y)
:=
\sum_{n\in\mathbb{Z}} ( -2n x +(1-2n)y) e^{xp_{2n} +y q_{2n}}
\label{eq:defhpq0}
\end{equation}
where $(p_{2n},q_{2n})$ is a recursive sequence of points on the parabola $\mathcal{P}$ starting from $(p_0,q_0)=(p(0),q(0))=(\gamma^2/2,-\gamma+\gamma^2/2)\in\partial\mathcal{A}$, and defined by~\eqref{eq:pnqn}, which means in this case
\begin{equation}
\begin{cases}
p_{2n}:=\gamma^2/2 -2n^2,
\\
q_{2n}:=\gamma^2/2 -\gamma -2n(n-1).
\end{cases}
\label{eq:pnqnalpha0}
\end{equation} 
Then, the function $\widetilde h^{(p_0,q_0)}$ is positive and harmonic for the process $Z$ killed at the boundary of the quadrant, that is satisfying~\eqref{eq:harm}. A symmetric expression holds for $\widetilde h^{(p(\pi/2),q(\pi/2))}$.
\label{prop:compensationharm0}
\end{proposition}
\begin{proof}
First, we again notice that the convergence of the sum is trivial since the sequence $(p_{2n},q_{2n})$ tends to minus infinity quadratically. 
Here we cannot compensate in the same way that in the previous proposition because we would obtain the null function since $(p_{2n},q_{2n})=(p_{-2n+1},q_{-2n+1})$. Therefore, we compensate using linear polynomials in front of the exponential terms.
With a straightforward computation we notice that a function $f(x,y)=(\lambda_1x+\lambda_2y)e^{xp+yq}$ satisfies $\mathcal{G}f=0$ if and only if $K(p,q)=0$ and $\lambda_1((1-\gamma)+p-q)+\lambda_2(\gamma+q-p)=0$. It is easy to verify that 
$K(p_{2n},q_{2n})=0$ and $-2n((1-\gamma)+p_{2n}-q_{2n})+(1-2n)(\gamma+q_{2n}-p_{2n})=0$.
Therefore we deduce that $\mathcal{G} \widetilde h^{(p_0,q_0)}  =0$ for $z=(x,y)\in\mathbb{R}_+^2$. Furthermore, using the fact that $(p_{2n},q_{2n})=(p_{-2n},q_{-2n+2})$ we have the following recursive compensations, we obtain that $\widetilde h^{(p(0),q(0))}(x,y)$ is equal to
\begin{multline*}
 \overunderbraces{ & \br{1}{=0 \text{ when } y=0} &\br{4}{=0  \text{ when } y=0 }& \br{4}{=0  \text{ when } y=0 }}%
  { & y e^{p_{0}x+q_{0}y} &- (2x+y)& e^{p_{2}x+q_{2}y}  &+&(2x+3y)e^{p_{-2}x-q_{-2}y} &-(4x+3y)&e^{p_{4}x+q_{4}y} &+(4x+5y) & e^{p_{-4}x+q_{-4}y} &+ \cdots}%
  {& \br{3}{=0  \text{ when } x=0 } & &\br{3}{=0  \text{ when } x=0 }}
\end{multline*}
and we deduce that $\widetilde h^{(p(0),q(0))}$ satisfies the boundary conditions of~\eqref{eq:harm}. Positivity is proved in the same way as in the proof of Proposition~\ref{prop:compensationharm}. It is easy to see that the function $\widetilde h^{(p(0),q(0))}$ is positive for $x+y$ large enough and we conclude in the same way thanks to the harmonic character of the function.
\end{proof}

The following proposition shows that the compensation approach actually
allows to construct all positive minimal harmonic functions.
This implies a correspondence between minimal positive harmonic functions and the arc of parabola $\overline{\mathcal{A}}$.
\begin{theorem}[Minimal harmonic functions]
For all $\alpha\in[0,\pi/2]$ we have
\begin{equation}
\widetilde h^{(p(\alpha),q(\alpha))} = h^\alpha
\label{eq:htildeh}
\end{equation}
where $(p(\alpha),q(\alpha))$ is the saddle point defined in \eqref{eq:saddlepoint}. The Martin boundary given in Theorem~\ref{cor:martinboundary} is minimal. It is homeomorphic to the arc of parabola $\overline{\mathcal{A}}$ through the map 
\renewcommand{\arraystretch}{1.7}
$$
\begin{array}{ccccc}
 \overline{\mathcal{A}} & \longrightarrow & \partial_M^{\mathbb{R}_+^2} Z \\
(p(\alpha),q(\alpha)) & \longmapsto & \displaystyle \frac{\widetilde h^{(p(\alpha),q(\alpha))}(z_0)}{\widetilde h^{(p(\alpha),q(\alpha))}(1,1)} . \\
\end{array}
$$
 See Figures~\ref{fig:pointcol} and~\ref{fig:seq}.
\label{prop:minimalharmonic}
\end{theorem}
\begin{proof} Recall that the Martin boundary $\partial_M^{\mathbb{R}_+^2} Z$ is determined in Theorem~\ref{cor:martinboundary} and is homeomorphic to $[0,\pi/2]$. In Propositions~\ref{prop:compensationharm} and~\ref{prop:compensationharm0} we have seen that the functions $\widetilde h^{(p(\alpha),q(\alpha))} $ are harmonic for $\alpha\in[0,\pi/2]$. By Martin boundary theory~\cite{Doob2001}, 
we have
\begin{equation}
\widetilde h^{(p(\alpha),q(\alpha))} (z_0) =\int_0^{\pi/2} \frac{h^\beta (z_0)}{h^\beta (1,1)} \mathrm{d}m_\alpha(\beta)
\label{eq:htildeMB}
\end{equation} 
for some measure $m_\alpha$ on $[0,\pi/2]$. Let $z_0=(x_0,y_0)$ and $\alpha\in(0,\pi/2)$, when $x_0\to\infty$ and $y_0\to\infty$ we have $\widetilde h^{(p(\alpha),q(\alpha))} (z_0)\sim e^{x_0 p(\alpha)+ y_0 q(\alpha)}$ by~\eqref{eq:defhpq} and $h^\beta (z_0) \sim e^{x_0 p(\beta)+y_0 q(\beta)}$ by Proposition~\ref{prop:conditionedescapeproba}. For $\alpha=0$, we have $\widetilde h^{(p(0),q(0))} (z_0)\sim y_0 e^{x_0 p(0)+ y_0 q(0)}$ by~\eqref{eq:defhpq0} and $h^\beta (z_0) \sim y_0 e^{x_0 p(\beta)+y_0 q(\beta)}$ and a symetrical result holds for $\alpha=\pi/2$. Then, from these asymptotics and~\eqref{eq:htildeMB} we can deduce that $m_\alpha =h^\alpha (1,1) \delta_\alpha$ (where $\delta_\alpha$ is the dirac measure at $\alpha$) and therefore $\widetilde h^{(p(\alpha),q(\alpha))} = h^\alpha$. Indeed, it is not possible to create an exponential asymptotic of a certain parameter by averaging exponential asymptotics of other parameters.
We deduce that these harmonic functions are minimal since there is a unique measure on the Martin boundary which satisfies~\eqref{eq:htildeMB}. The Martin boundary is then equal to the minimal Martin boundary and is homeomorphic to the arc of parabola $\overline{\mathcal{A}}$ through the composition between the map $\Phi$ given in Theorem~\ref{cor:martinboundary} and the homeomorphism $[0,\pi/2]\to \overline{\mathcal{A}} :\alpha \mapsto (p(\alpha),q(\alpha))$.
\end{proof}

The same kind of phenomenon appears for non-singular random walks~\cite{ignatiouk_loree_2010} and singular one~\cite{hoang_raschel_tarrago_harmonic_22}.

\subsection{Persistence and exit probabilities}

From~\eqref{eq:escapeprobah},~\eqref{eq:defhpq} and~\eqref{eq:htildeh} we find explicitly the persistence probability of the conditioned process: $\mathbb{P}^\alpha(T=\infty)=e^{-x_0 p(\alpha)-y_0 q(\alpha)}  h^\alpha(z_0)$. The next remark focuses on the non-conditioned process. 

\begin{remark}[Persistence probability]
Let $\alpha_\gamma \in (0,\pi/2)$ such that $\tan\alpha_\gamma=\gamma/(1-\gamma)$, we have $(p(\alpha_\gamma),q(\alpha_\gamma))=(0,0)$ and then $\mathbb{P}=\mathbb{P}^{\alpha_\gamma}$. Then, the probability 
that the process stays in the cone forever is equal to
$$
\mathbb{P}_{(t_0,y_0)}(\forall t>0 : X_t\in C)=\mathbb{P}(T=\infty)=\widetilde h^{(0,0)}(z_0)= h^{\alpha_\gamma} (z_0) .
$$
For $z_0=(x_0,y_0)=(t_0-y_0,y_0)$ and remembering equation~\eqref{eq:pnqn} we obtain
\begin{align*}
\mathbb{P}(T=\infty)
&=\sum_{n\in\mathbb{Z}} (-1)^n e^{p_n x_0 + q_n y_0}
=
\sum_{n\in\mathbb{Z}} e^{p_{2n} x_0 + q_{2n} y_0} - \sum_{n\in\mathbb{Z}} e^{\overbrace{p_{2n+1}}^{p_{2n}} x_0 + \overbrace{q_{2n+1}}^{q_{2n+2}} y_0}
\\ &=
\sum_{n\in\mathbb{Z}} 2 \sinh \left( \frac{q_{2n} - q_{2n+2}}{2} y_0 \right) e^{p_{2n} x_0 + \frac{q_{2n}+ q_{2n+2}}{2} y_0 }
\\ &= 2 \sum_{n\in\mathbb{Z}} \sinh((2n+\gamma)y_0) e^{-2n(n+\gamma)t_0-\gamma y_0}.
\end{align*}
We can also express this quantity in another way to match Defosseux's result~\cite[Proposition 2.3]{defosseux_2016}. By noting that $p_{-2n} x_0 + q_{-2n} y_0=-2n^2t_0+2n\gamma t_0-2ny_0$ and ${p_{2n}} x_0 +{q_{2n+2}} y_0=-2n^2t_0-2n\gamma t_0-2\gamma y_0-2ny_0$ 
and with a simple change of variable in the sum we obtain
\begin{align*}
\mathbb{P}(T=\infty)
&=
\sum_{n\in\mathbb{Z}} e^{p_{-2n} x_0 + q_{-2n} y_0} - \sum_{n\in\mathbb{Z}} e^{{p_{2n}} x_0 +{q_{2n+2}} y_0}
\\ &=  2 \sum_{n\in\mathbb{Z}} \sinh(\gamma (y_0+2nt_0)) e^{-2(n y_0+n^2 t_0)-\gamma y_0}.
\end{align*}
\label{rem:persistproba}
\end{remark}

We are now interested in $\mathbb{P}^\alpha(T_1<T_2)$ and $\mathbb{P}^\alpha(T_2<T_1)$ which are the exit probabilities on each side of the cone of the conditioned process. 
\begin{proposition}[Conditioned probabilities of exit on an edge]
For $\alpha\in [0,\pi]$ we have
\begin{align*}
L_{1}(q(\alpha))
= e^{p(\alpha) x_0+q(\alpha)y_0}\mathbb{P}^\alpha(T_1<T_2)
=\sum_{n\in -\mathbb{N}^*} (-1)^{n+1} e^{p_n x_0 + q_n y_0}
\end{align*}
and 
\begin{align*}
L_2(p(\alpha))
= e^{p(\alpha) x_0+q(\alpha)y_0}\mathbb{P}^\alpha(T_2<T_1)
=\sum_{n\in +\mathbb{N}^*} (-1)^{n+1} e^{p_n x_0 + q_n y_0}
\end{align*}
where $(p_0,q_0)=(p(\alpha),q(\alpha))$ and the sequence $(p_n,q_n)_{n\in\mathbb{Z}}$ is defined in~\eqref{eq:pnqn}.
\label{prop:L1explicit}
\end{proposition}
\begin{proof} The proof is again based on a compensation method.
Let us define
$$\widetilde L_1^\alpha (z_0):=\sum_{n\in -\mathbb{N}^*} (-1)^{n+1} e^{p_n x_0 + q_n y_0}.$$ 
Let us verify that the function $\widetilde L_1^\alpha$ is harmonic and satisfies the following conditions:
\begin{equation}
\begin{cases}
(\mathcal{G}  \widetilde L_1^\alpha) (z_0)  =0 
& \text{for } z_0\in(\mathbb{R}_+^*)^2,
\\
 \widetilde L_1^\alpha (z_0) =0 & \text{for } z_0=(x_0,0)\in \mathbb{R}\times \{0\} ,
\\
 \widetilde L_1^\alpha (z_0) =e^{ q(\alpha)y_0} & \text{for } z_0=(0,y_0)\in  \{0\} \times \mathbb{R}.
\end{cases}
\label{eq:uharm}
\end{equation}
The first condition is trivial by the construction of the sequel $(p_n,q_n)\in \mathcal{P}$, see Proposition~\ref{prop:compensationharm}. The boundary conditions can be seen by summing in packets of two (taking into account the first term which cannot be compensated on one side) as follows
$$
\widetilde L_1^\alpha (z_0)=  \overunderbraces{& \br{1}{=e^{ q(\alpha)y_0}  \text{ when } x_0=0 } &\br{4}{=0  \text{ when } x_0=0 }& \br{4}{=0  \text{ when } x_0=0 }}%
  {&e^{p_{-1}x_0+q_{-1}y_0} &-& e^{p_{-2}x_0+q_{-2}y_0}  &+&e^{p_{-3}x_0-q_{-3}y_0} &-&e^{p_{-4}x_0+q_{-4}y_0} &+ & e^{p_{-5}x_0+q_{-5}y_0} &+ \cdots}%
  {& \br{3}{=0  \text{ when } y_0=0 } & &\br{3}{=0  \text{ when } y_0=0 }}.
$$
With a classical martingale argument we remark that the function $z_0\mapsto L_1(q(\alpha))$ seen as a function of $z_0$ is also harmonic, i.e. $\mathcal{G}_{z_0}  L_1(q(\alpha))=0$. This function also satisfies the same boundary conditions, to verify it, it is enough to remark that via the Doob's $h$-transform we have
$$L_1(q(\alpha))
= e^{p(\alpha) x_0+q(\alpha)y_0} L_1^\alpha(0)
= e^{p(\alpha) x_0+q(\alpha)y_0}\mathbb{P}_{z_0}^\alpha(T_1<T_2)$$
and of course $\mathbb{P}^\alpha(T_1<T_2)=0$ if $z_0\in \mathbb{R}\times \{0\} $ and $\mathbb{P}^\alpha(T_1<T_2)=1$ if $z_0\in  \{0\} \times \mathbb{R} $.
The functions $\widetilde L_1^\alpha$ and $L_1(q(\alpha))$ seen as a function of $z_0$ both satisfy~\eqref{eq:uharm} and tend to $0$ at infinity. We deduce by the maximum principle that these functions are equals, i.e. $L_1(q(\alpha))=\widetilde L_1^\alpha$. 
\end{proof}

\begin{remark}[Probability of exit on an edge]
For $\alpha_\gamma \in (0,\pi/2)$ such that $\tan\alpha_\gamma=\gamma/(1-\gamma)$ we have $(p(\alpha_\gamma),q(\alpha_\gamma))=(0,0)$ and we find that the probability $\mathbb{P}(T_1<T_2)$ 
that the process exit of the cone by the vertical axis is equal to
\begin{align*}
\mathbb{P}(T_2<T_1)=L_2(0)
&=\sum_{n\in\mathbb{N}^*} (-1)^{n+1} e^{p_n x_0 + q_n y_0}
=
\sum_{n\in\mathbb{N}} e^{\overbrace{p_{2n+1}}^{p_{2n}} x_0 + \overbrace{q_{2n+1}}^{q_{2n+2}} y_0} - \sum_{n\in\mathbb{N}} e^{p_{2n+2} x_0 + {q_{2n+2}} y_0}
\\ &=
\sum_{n\in\mathbb{N}} 2 \sinh \left( \frac{p_{2n} - p_{2n+2}}{2} x_0 \right) e^{ \frac{p_{2n}+ p_{2n+2}}{2} x_0 + q_{2n+2} y_0}
\\ &= 2 \sum_{n\in\mathbb{N}} \sinh((2n+1+\gamma)x_0) e^{-(2n+2)(n+\gamma)t_0 -(\gamma+1) x_0}
\end{align*}
where $(p_0,q_0)=(0,0)$.
Similarly, we obtain
$$
\mathbb{P}(T_1<T_2)= 2 \sum_{n\in\mathbb{N}} \sinh((2n+2-\gamma)y_0) e^{-(2n+2)(n+1-\gamma)t_0 -(2-\gamma) y_0}.
$$
This result matches with Anderson's result \cite[Thm 4.1]{anderson_1960}.
\end{remark}

\begin{corollary}[Laplace transform on the boundary]
The Laplace transform $L_1$ can be extended analytically on $\mathbb{C}\setminus [q^+,\infty)$ by the formula
\begin{align*}
&L_1(q)  \\
&=
\sum_{n\in -\mathbb{N}^*} 2 \sinh \left( (2n-p_0-q_0+\gamma+1) x_0 \right) e^{ (p_0+(2n+1)(p_0-q_0)-2n(n+\gamma+1)-\gamma-1) x_0 + (q+(2n+2)(P^+(q)-q)-(2n+2)(n+\gamma)) y_0}
\\ &=
\sum_{n\in -\mathbb{N}^*}  
 e^{(p_0 +2n(p_0-q_0)-2n(n+\gamma) ) x_0 + (q_0 +(2n+2)(p_0-q_0)-(2n+2)(n+\gamma)) y_0}
\\
& \quad - e^{(p_0 +(2n+2)(p_0-q_0)-(2n+2)(n+1+\gamma) ) x_0 + (q_0 +(2n+2)(p_0-q_0)-(2n+2)(n+\gamma)) y_0}
\end{align*}
where $(p_0,q_0)=(P^+(q),q)$ and
\begin{equation}
\begin{cases}
p_{2n}=P^+(q)+2n(P^+(q)-q)-2n(n+\gamma)
\\
q_{2n}=q+2n(P^+(q)-q)-2n(n+\gamma-1)
\end{cases}
\quad\text{and}\quad
(p_{2n+1},q_{2n+1})=(p_{2n},q_{2n+2}).
\label{eq:pnqnP+}
\end{equation}
\label{cor:L1explicit}
\end{corollary}
\begin{proof}
The result directly derives from Proposition~\ref{prop:L1explicit}, the fact that $(p(\alpha),q(\alpha))=(P^+(q(\alpha)),q(\alpha))$ and the analyticity of $P^+$ on $\mathbb{C}\setminus [q^+,\infty)$. Taking $(p_0,q_0)=(P^+(q),q)$ we obtain that
\begin{align*}
L_1(q) &=\sum_{n\in -\mathbb{N}^*} (-1)^{n+1} e^{p_n x_0 + q_n y_0}
=
\sum_{n\in -\mathbb{N}^*}  
(e^{p_{2n} x_0 }
- e^{p_{2n+2} x_0 })e^{ q_{2n+2} y_0}
\\ &=
\sum_{n\in -\mathbb{N}^*} 2 \sinh \left( \frac{p_{2n} - p_{2n+2}}{2} x_0 \right) e^{ \frac{p_{2n}+ p_{2n+2}}{2} x_0 + q_{2n+2} y_0}
\end{align*}
and we conclude noticing that~\eqref{eq:pnqn} implies~\eqref{eq:pnqnP+}.
\end{proof}

\section{Exit densities and transition kernel}
\label{sec:transitionkernel}


This last section aims at computing the inverse Laplace transforms of $L_1$ and $L_2$ to obtain an explicit expression for the exit densities $f_1$ and $f_2$. Finally, these expressions are used to calculate the kernel $p_{(t_0,y_0)}^{k,C}(t,y)$.

\begin{theorem}[Exit densities on the boundaries] Let $z_0=(x_0,y_0)$ the starting point and $t_0=x_0+y_0$. Exit densities are given by the following formulas:
\begin{align*}
f_1(y)
=
\frac{e^{-t_0 (1-\gamma)-y_0 \gamma}}{\sqrt{2\pi}}
\frac{e^{-\frac{1}{2} (1-\gamma)^2(y-t_0) }}{\sqrt{(y-t_0)^3}} 
\sum_{n\in \mathbb{Z}}
((2n+1)t_0-y_0) e^{-\frac{1}{2} \frac{((2n+1)t_0-y_0)^2}{y-t_0} +(2n+1)y_0-2n(n+1)t_0}
\end{align*}
\begin{align*}
f_2(x)
=
\frac{e^{-t_0 \gamma-x_0(1- \gamma)}}{\sqrt{2\pi}}
\frac{e^{-\frac{1}{2} \gamma^2(x-t_0) }}{\sqrt{(x-t_0)^3}} 
\sum_{n\in \mathbb{Z}}
((2n+1)t_0-x_0) e^{-\frac{1}{2} \frac{((2n+1)t_0-x_0)^2}{x-t_0} +(2n+1)x_0-2n(n+1)t_0}
\end{align*}
\label{thm:f1f2}
\end{theorem}
\begin{proof}
We seek to inverse the Laplace transform of Corollary~\ref{cor:L1explicit}. We have
\begin{align*}
L_1(q) &= 
\sum_{n\in \mathbb{N}}  
(e^{p_{-2n-2} x_0 +q_{-2n} y_0 }
- e^{p_{-2n-2} x_0 +q_{-2n-2} y_0 })
\\ &= \sum_{n\in \mathbb{N}}
e^{-(P^+(q)-q)((2n+1)t_0-y_0)+q t_0- (2n+2)(n+1-\gamma) t_0 + 2(n+1-\gamma)y_0}
\\ & \quad -
e^{-(P^+(q)-q)((2n+1)t_0+y_0)+q t_0- (2n+2)(n+1-\gamma) t_0 - (2n+2)y_0}
\\ &= \sum_{n\in \mathbb{N}}
e^{-((2n+1)t_0-y_0)\sqrt{(1-\gamma)^2-2q}+q t_0- t_0(2n(n+1)+1-\gamma)+y_0(2n+1-\gamma)}
\\ & \quad -
e^{-((2n+1)t_0+y_0)\sqrt{(1-\gamma)^2-2q}+q t_0- t_0(2n(n+1)+1-\gamma)-y_0(2n+1+\gamma)}
\end{align*}
We denote $\mathcal{L}^{-1}_q$ the inverse Laplace operator related to the $q$-variable.
For $a>0$, $b>0$ and $y>t_0$ we have
$$
\mathcal{L}^{-1}_q (e^{-b\sqrt{a-2q}+q t_0})(y)= \frac{b}{\sqrt{2\pi}}\frac{1}{\sqrt{(y-t_0)^3}}e^{-\frac{1}{2} (a(y-t_0) +\frac{b^2}{y-t_0})}.
$$
Taking the inverse Laplace transform of the sum term by term, it reads
\begin{align*}
f_1(y) &= 
\frac{e^{-t_0 (1-\gamma)-y_0 \gamma}}{\sqrt{2\pi}}\frac{1}{\sqrt{(y-t_0)^3}}
\\ & \sum_{n\in \mathbb{N}}
((2n+1)t_0-y_0) e^{-\frac{1}{2} ((1-\gamma)^2(y-t_0) +\frac{((2n+1)t_0-y_0)^2}{y-t_0})}
e^{- 2n(n+1)t_0+(2n+1)y_0}
\\ & \quad - 
((2n+1)t_0+y_0) e^{-\frac{1}{2} ((1-\gamma)^2(y-t_0) +\frac{((2n+1)t_0+y_0)^2}{y-t_0})}
e^{- 2n(n+1)t_0-(2n+1)y_0}
\end{align*}
which coincides with the stated result summing over $\mathbb{Z}$. The formula for $f_2(x)$ is obtained replacing $1-\gamma$ by $\gamma$ and $y$ by $x$.
\end{proof}

With some tedious computations, one may verify that this result matches with Anderson's result \cite[Thm 4.3]{anderson_1960} by deriving its formula (4.32).
\begin{remark}[Density of the first exit time and persistence asymptotics]
The previous result gives an explicit expression for the density of $T$ since we have seen in~\eqref{eq:fT} that $f_T(t)=f_1(t_0+t)+f_2(t_0+t)$. The persistence probability of the process after a time $t$ is given by 
\begin{align*}
\mathbb{P}(T>t &) =\mathbb{P}(T=\infty)+\int_t^\infty f_T(s)\mathrm{d}s
\\ &
\underset{t\to\infty}{=}
h^{\alpha_\gamma}(z_0) + 
\left( \frac{ h^{0}(z_0)}{\sqrt{2\pi}} \frac{e^{-\frac{\gamma^2}{2}(t+t_0)}}{t^{3/2}} + \frac{ h^{\pi/2}(z_0)}{\sqrt{2\pi}} \frac{e^{-\frac{(1-\gamma)^2}{2}(t+t_0)}}{t^{3/2}}  \right)(1+o(1))
\end{align*}
where the asymptotics derives from~\eqref{eq:asymptboundaries} and $\alpha_\gamma$ is defined in Remark~\ref{rem:persistproba}.
\end{remark}

We now use the previous result about the exit densities 
to obtain an 
explicit formula for the Green's function of the process $Z$, and then for the transition density of the space-time Brownian motion $B$, see \eqref{eq:green=densitytransition}.

\begin{corollary}[Transition density of the killed space-time Brownian motion in a cone]
For $(t_0+t,y)\in C$, the transition density defined in~\eqref{eq:defpkC} is equal to
\begin{align}
&p_{(t_0,y_0)}^{k,C}(t,y) =\frac{1}{\sqrt{2\pi t}} e^{-\frac{(y-y_0-\gamma t)^2}{2t}}- \mathbb{E}_{(t_0,y_0)}\left( \frac{1}{\sqrt{2\pi (t-T)}} e^{-\frac{(y-Y_T-\gamma (t-T))^2}{2(t-T)}}\mathds{1}_{T\leqslant t}\right)
\label{eq:pkC1}
\\
 &=
\frac{1}{\sqrt{2\pi t}} \sum_{n\in \mathbb{Z}} 
\left(
e^{-\frac{(y-2n t_0-y_0-\gamma t)^2}{2t}+\gamma (2n t_0 +y_0)}
-
e^{-\frac{(y+2n t_0+y_0-\gamma t)^2}{2t} - \gamma (2n t_0 +y_0)}
\right)
 e^{-\gamma y_0 -2n^2t_0 -2ny_0   }
 \label{eq:pkC2}
\end{align}
\label{cor:transitionkernel}
\end{corollary}
\begin{proof}
Let us define $p_t^\gamma(y):=\frac{1}{\sqrt{2\pi t}} e^{-\frac{(y-\gamma t)^2}{2t}}$ the transition kernel of the free Brownian motion with drift $\gamma$ starting from $0$ and recall that we note $p^{k,C}_{(t_0,y_0)}(t,y) \mathrm{d}y 
=\mathbb{P}_{(t_0,y_0)}(B(t)=(t_0+t,\mathrm{d}y), T>t)$ the transition kernel of the killed space-time Brownian motion in $C$. 
A direct calculus based on the Markov property gives
\begin{align*}
p^{k,C}_{(t_0,y_0)}(t,y) \mathrm{d}y 
&=
\mathbb{P}_{(t_0,y_0)}(B(t)=(t_0+t,\mathrm{d}y))-
\mathbb{P}_{(t_0,y_0)}(B(t)=(t_0+t,\mathrm{d}y), T\leqslant t)
\\ &=
\mathbb{P}_{(t_0,y_0)}(B(t)=(t_0+t,\mathrm{d}y))
-
\int_0^t
\mathbb{P}_{(t_0,y_0)}(B(t)=(t_0+t,\mathrm{d}y) | T=T_2=u)f_2(t_0+u)\mathrm{d}u
\\ &-
\int_0^t
\mathbb{P}_{(t_0,y_0)}(B(t)=(t_0+t,\mathrm{d}y) | T=T_1=v)f_1(t_0+v)\mathrm{d}v
\\ &=
\mathbb{P}_{(t_0,y_0)}(B(t)=(t_0+t,\mathrm{d}y))-
\int_0^t \mathbb{P}_{(t_0+u,0)}(B(t)=(t_0+t,\mathrm{d}y))f_2(t_0+u)\mathrm{d}u
\\ &-
\int_0^t \mathbb{P}_{(t_0+v,t_0+v)}(B(t)=(t_0+t,\mathrm{d}y))f_1(t_0+v)\mathrm{d}v
\\ &= \left( p_t^\gamma(y-y_0)
-
\underbrace{ \int_0^t p_{t-u}^\gamma(y)f_2(t_0+u)\mathrm{d}u}_{=:I_2(t)}
-
\underbrace{\int_0^t p_{t-v}^\gamma(y-t_0-v)f_1(t_0+v)\mathrm{d}v }_{=:I_1(t)}
\right) \mathrm{d}y
\end{align*}
and equality~\eqref{eq:pkC1} follows.
We defined in the previous formula the integrals $I_1(t)$ and $I_2(t)$ and we are going to compute these integrals by computing their Laplace transforms. We denote $\mathcal{L}_t$ the Laplace transform operator related to the $t$ variable and  $\star$ the convolution operation. We have
\begin{align}
\nonumber \mathcal{L}_t(I_2(t))(p) &=
\mathcal{L}_t \left( 
 \left( p_{\cdot}^\gamma(y)  \star f_2(t_0+\cdot)\right) (t) \right) (p)
\\
 &= 
 \mathcal{L}_t \left( 
 p_{t}^\gamma(y) \right) (p)  \times \mathcal{L}_t \left( f_2(t_0+t)  \right) (p)
\label{eq:LI2}
\end{align}
One may compute that
\begin{equation}
\mathcal{L}_t \left(   
 p_{t}^\gamma(y)
 \right) (p)
=\frac{1}{\sqrt{2(p+\gamma^2/2)}} e^{\gamma y-y\sqrt{2(p+\gamma^2/2)}}.
\label{eq:Lpgamma}
\end{equation}
Recalling that Theorem~\ref{thm:f1f2} gives
$$
f_2(t_0+u)
=
\frac{e^{-t_0 \gamma-x_0(1- \gamma)}}{\sqrt{2\pi}}
\frac{e^{-\frac{1}{2} \gamma^2u }}{{u^{3/2}}} 
\sum_{n\in \mathbb{Z}}
((2n+1)t_0-x_0) e^{-\frac{1}{2} \frac{((2n+1)t_0-x_0)^2}{u} +(2n+1)x_0-2n(n+1)t_0}
$$
we compute the following Laplace transform (we note $sgn$ the sign function)
$$
\mathcal{L}_t \left( 
\frac{((2n+1)t_0-x_0)}{\sqrt{2\pi}{t^{3/2}}} e^{-\frac{1}{2} \gamma^2t -\frac{1}{2} \frac{((2n+1)t_0-x_0)^2}{t}}   \right) (p)
= sgn(n) e^{-|(2n+1)t_0-x_0|\sqrt{2(p+\gamma^2/2)}}
$$
and we obtain
\begin{equation}
\mathcal{L}_t \left( 
 f_2(t_0+t)  \right) (p)=
 {e^{-t_0 \gamma-x_0(1- \gamma)}} 
\sum_{n\in \mathbb{Z}}
sgn(n)
 e^{-sgn(n)((2n+1)t_0-x_0)\sqrt{2(p+\gamma^2/2)} +(2n+1)x_0-2n(n+1)t_0}.
 \label{eq:Lf2}
\end{equation}
Using \eqref{eq:LI2}, \eqref{eq:Lpgamma} and \eqref{eq:Lf2} we obtain the Laplace transform of $I_2$. By inverting this Laplace transform we compute
\begin{align*}
I_2(t) &=  {e^{-t_0 \gamma-x_0(1- \gamma)}} 
\sum_{n\in \mathbb{Z}} sgn(n)
p_t^\gamma (y+sgn(n)((2n+1)t_0-x_0))
 e^{-sgn(n)\gamma((2n+1)t_0-x_0) +(2n+1)x_0-2n(n+1)t_0}
 \\ &=
 {e^{-\gamma y_0 }} 
\sum_{n\in \mathbb{Z}} sgn(n)
p_t^\gamma (y+sgn(n)(2n t_0+y_0))
 e^{-2n^2t_0 -2ny_0 - sgn(n) \gamma (2n t_0 +y_0)  }
 \\ &=
 {e^{-\gamma y_0 }} 
\sum_{n\in \mathbb{N}} 
p_t^\gamma (y+2n t_0+y_0)
 e^{-2n^2t_0 -2ny_0 - \gamma (2n t_0 +y_0)  }
 \\
 & - 
 {e^{-\gamma y_0 }} \sum_{n\in -\mathbb{N}^*} 
p_t^\gamma (y-2n t_0-y_0)
 e^{-2n^2t_0 -2ny_0 + \gamma (2n t_0 +y_0)  }
\end{align*}
In the same way, we show that for $x+y=t_0+t$ we have
\begin{align*}
I_1(t)
&=
 {e^{-(1-\gamma) x_0 }} 
\sum_{n\in \mathbb{N}} 
p_t^{1-\gamma} (x+2n t_0+x_0)
 e^{-2n^2t_0 -2nx_0 - (1-\gamma) (2n t_0 +x_0)  }
 \\
 & - 
 {e^{-(1-\gamma) x_0 }} \sum_{n\in -\mathbb{N}^*} 
p_t^{1-\gamma} (x-2n t_0-x_0)
 e^{-2n^2t_0 -2nx_0 + (1-\gamma) (2n t_0 +x_0)  }
 \\
 &=
 {e^{-\gamma y_0 }} 
\sum_{n\in \mathbb{N}} 
p_t^\gamma (y-2(n+1) t_0+y_0)
 e^{-2(n+1)^2t_0 +2(n+1)y_0 +\gamma (2(n+1) t_0 -y_0)  }
 \\
 & - 
 {e^{-\gamma y_0 }} \sum_{n\in -\mathbb{N}^*} 
p_t^\gamma (y+2n t_0-y_0)
 e^{-2n^2t_0 +2ny_0 - \gamma (2n t_0 -y_0)  }
 \\ &=
 {e^{-\gamma y_0 }} 
\sum_{n\in -\mathbb{N}^*} 
p_t^\gamma (y+2n t_0+y_0)
 e^{-2n^2t_0 -2ny_0 - \gamma (2n t_0 +y_0)  }
 \\
 & - 
 {e^{-\gamma y_0 }} \sum_{n\in \mathbb{N}^*} 
p_t^\gamma (y-2n t_0-y_0)
 e^{-2n^2t_0 -2ny_0 + \gamma (2n t_0 +y_0)  }
\end{align*}
Since $p^{k,C}_{(t_0,y_0)}(t,y)=p_t^\gamma(y-y_0)-I_1(t)-I_2(t)$ we obtain formula~\eqref{eq:pkC2} which concludes the proof.
\end{proof}
We find again the result obtained by Defosseux \cite[Proposition 2.2]{defosseux_2016}. 

The approach developed in this article which combines the steepest descent method and a recursive compensation method allowed us to determine the Martin boundary and all the harmonic functions associated with the space-time Brownian motion. By using these two methods together, we also developed a new and original approach which enabled us to recover some interesting results.

\subsection*{Acknowledgments} 
This project has received funding from Agence Nationale de la Recherche, ANR JCJC programme under the Grant Agreement RESYST ANR-22-CE40-0002.
The author would like to thank Manon Defosseux, Kilian Raschel and Pierre Tarrago for interesting discussions related to space-time Brownian motion, the compensation approach and Martin's boundary theory.

\newpage
\bibliographystyle{apalike}



\end{document}